\documentclass[12pt]{amsart}

\usepackage{amsmath, amssymb, amsthm}
\usepackage{mathrsfs}
\usepackage{hyperref}

\newtheorem{proposition}{Proposition}
\newtheorem{theorem}{Theorem}
\newtheorem{lemma}{Lemma}
\newtheorem{corollary}[theorem]{Corollary}
\theoremstyle{remark}
\newtheorem{remark}{Remark}
\theoremstyle{definition}
\newtheorem{definition}{Definition}[section]

\newcommand{\R}{\mathbb R}

\title[Slicing Support Function]{Slicing Support Functions with Recovery Formula and Curvature Identities}

\author{Yen-Chang Huang}

\address{Department of Applied Mathematics, National Yang Ming Chiao Tung University, Hsinchu City, 30010, Taiwan}

\email{ychuang0802@nycu.edu.tw}
\thanks{The work was supported by the National Science and Technology Council (NSTC), Taiwan, under grant number NSTC 112-2115-M-A49-014-MY3 and 115-2115-M-A49-001-MY2}

\begin{document}
\begin{abstract}
Let $K\subset\mathbb{R}^n$ be a convex body with support function $h_K$.
For $\nu\in\mathbb{S}^{n-1}$, $p\in\mathbb{R}$, and $u\in\nu^\perp$, we
introduce the slicing support function $h_\nu(u,p)$, defined as the support
function of the slice $K\cap\{x\cdot\nu=p\}$ in the direction $u$. For each
fixed $p$, this is precisely the support function of the corresponding
translated fiber appearing in the construction of the convex fiber body of
Mathis and Meroni \cite{MathisMeroni2023}. We derive an infimal
representation of $h_\nu$ in terms of $h_K$, together with a corresponding
minimax identity. Using the Fenchel--Moreau theorem, we prove that $h_K$,
and hence $K$, can be recovered from the slicing support function without
any regularity assumption on $\partial K$. We also obtain a differential
recovery formula when $K$ is strictly convex and $\partial K$ is of class
$C^1$.

In dimension three, we establish a cylindrical Monge--Amp\'{e}re-type
determinant identity expressed in terms of the spherical curvature matrix of
$\partial K$. When the relevant tangent directions are principal directions,
this determinant reduces to a weighted ratio of the corresponding principal
radii of curvature. We further characterize this principal-direction
condition by showing that, for convex bodies with $C^2$-boundary and positive
Gaussian curvature, the spherical coordinate directions are principal
directions away from the poles if and only if, up to translation, the body is
a body of revolution. Finally, we extend the construction to
higher-codimensional iterated slicing support functions and derive a
full-Hessian determinant identity via the Schur complement.
\end{abstract}
\maketitle

\section{Introduction}\label{sec1}
Let \(K\subset\mathbb{R}^n\) be a convex body. Its support function is defined by
\[
h_K(v)=\sup_{x\in K} x\cdot v,
\qquad v\in\mathbb{R}^n.
\]
It is positively homogeneous of degree one and is therefore determined by its
restriction \(\bar h_K:=h_K|_{\mathbb{S}^{n-1}}\). Support functions provide a
standard analytic representation of convex bodies. If \(\partial K\) is of
class \(C^2\) and has positive Gauss curvature, then the spherical curvature
matrix
\begin{align}\label{sphericalcurvature1}
Q_{\bar h_K}
=
\nabla_{\mathbb{S}^{n-1}}^2\bar h_K+\bar h_K I
\end{align} is positive definite, and its eigenvalues are the principal radii of curvature
of \(\partial K\), where $\nabla^2_{\mathbb{S}^{n-1}}$ is the spherical Hessian. This matrix plays a fundamental role in the classical
Minkowski problem and its Monge--Amp\`ere formulation; see
\cite{schneider,ChengYau1976,CaffarelliNirenbergSpruck1984,TrudingerWang2008}.

Sections and projections of convex bodies form another central theme in convex geometry and geometric tomography. Gardner's geometric tomography studies the recovery of convex bodies from lower-dimensional information such as projections or cross-sections \cite{Gardner}. Many classical problems describe sections through scalar quantities, such as their volumes or associated section functions, leading to fundamental results including the Busemann--Petty problem and its variants; see Ball~\cite{Ball1988}, Gardner--Koldobsky--Schlumprecht~\cite{GardnerKoldobskySchlumprecht1999}, and Koldobsky~\cite{Koldobsky}. More recent work has developed related themes involving fibers, sections, and Monge--Amp\'{e}re structures. For instance, Hoehner and Xing studied new fiber-type combinations of convex bodies \cite{HoehnerXing2025}, while recent results on sections and projections have continued to investigate volume-related problems \cite{Tziotziou2026,MyroshnychenkoTangTatarkoTkocz2026,HaddadRyabogin2026}. Connections among support functions, geometric measures, and Monge--Amp\'{e}re equations have also been studied in \cite{HugMussnigUlivelli2026,HuangYangZhang2025}.

However, reducing a section to scalar data  discards much of its geometric structure. In particular, information describing how the section is positioned and shaped within the ambient body is largely lost. The present paper takes a different perspective: instead of providing a single value to each section, we preserve directional information on every slice. This leads to the notion of the slicing support function, which records finer geometric information about the family of sections and allows us to recover the support function of the original convex body through an optimization-based reconstruction formula (see \eqref{recov1} and \eqref{recov2} below).

Fiber polytopes were introduced by Billera and Sturmfels \cite{BilleraSturmfels1992}, and the construction was later extended to general convex bodies by Mathis-Meroni \cite{MathisMeroni2023}. More precisely, let $V \subset \mathbb{R}^{a+b}$ be an $a$-dimensional subspace, let $\pi : \mathbb{R}^{a+b} \to V$ be the orthogonal projection, and write $\mathbb{R}^{a+b} = V \oplus V^\perp$. For a convex body $K \subset \mathbb{R}^{a+b}$ and $x \in \pi(K)$, the fiber over $x$ is
\[
K_x := \{ y \in V^\perp : x+y \in K \};
\]
see \cite{MathisMeroni2023} for more detail. A measurable map $\gamma : \pi(K) \to V^\perp$ is called a section if $\gamma(x) \in K_x$ for every $x \in \pi(K)$. The fiber body of $K$ with respect to $\pi$ is then defined by
\[
\Sigma_\pi(K) := \left\{ \int_{\pi(K)} \gamma(x) \, dx : \gamma \text{ is a section} \right\},
\]
where the integral is in the sense of Minkowski-Aumann \cite{aumann1965integrals, minkowski1989volumen}, and the integration is taken with respect to the Lebesgue measure on $V$.

In this paper, we investigate the support function defined on the translated fiber $K_p^\nu$ (defined below) and its fundamental properties and connection with Monge–Amp\'ere-type equations. Given a convex body $K\subset \mathbb{R}^n$ and fix $\nu \in \mathbb{S}^{n-1}$. Set
\begin{align}\label{height}
m := \min_{x \in K} x \cdot \nu,  \qquad M:= \max_{x \in K} x \cdot \nu.
\end{align}
Consider the affine slice
\[
C(\nu, p) := K \cap \{ x \in \mathbb{R}^n : x \cdot \nu = p \},
\]
and its translation to the orthogonal complement $\nu^\perp$,
\[
K_p^\nu := \{ y \in \nu^\perp : p\nu + y \in K \}.
\]
Note that $C(\nu, p)\neq \emptyset$ when $p\in [m, M]$ and in this case $C(\nu, p) = p\nu + K_p^\nu$. We define, for any $u \in \nu^\perp$, the \emph{slicing support function}
\begin{equation}\label{intro0}
h_\nu(u, p) := \sup_{x \in C(\nu,p)} x \cdot u = \sup_{y \in K_p^\nu} y \cdot u = h_{K_p^\nu}(u).
\end{equation}
Thus, for each fixed $p$, $h_\nu(\cdot, p)$ is precisely the support function of the translated fiber $K_p^\nu$. Note that, by definition, the fiber body \(\Sigma_\pi(K)\) combines the entire family of fibers through Minkowski--Aumann integration, whereas the slicing support function records the geometry of each fiber separately. For instance, under suitable regularity assumptions, the boundary of \(C(\nu,p)\) can be recovered from \(h_\nu(\cdot,p)\) and its derivatives.

Let us relate the slicing support function to the fiber body associated with the projection $\pi_\nu$. Let $\Sigma_{\pi_\nu}(K) \subset \nu^\perp$ denote the fiber body of $K$ associated with $\pi_\nu(x) = (x \cdot \nu)\nu$ for all $x\in \mathbb{R}^n$.
Identifying $\operatorname{span}\{\nu\}$ with $\mathbb{R}$ through $p\nu \leftrightarrow p$, and following the unnormalized convention of \cite{MathisMeroni2023}, the  fiber body $\Sigma_{\pi_\nu}(K)$ is given by the Minkowski-Aumann integral
\[
\Sigma_{\pi_\nu}(K) = \int_{m}^{M} K_p^\nu \, dp.
\]
By \cite[Proposition~2.7]{MathisMeroni2023}, its support function satisfies
\[
h_{\Sigma_{\pi_\nu}(K)}(u) = \int_{m}^{M} h_{K_p^\nu}(u) \, dp, \qquad u \in \nu^\perp.
\]
In view of \eqref{intro0}, \(h_\nu\) and \(\Sigma_{\pi_\nu}(K)\) are related by the identity
\begin{equation}\label{eq:fiber-body-slicing}
h_{\Sigma_{\pi_\nu}(K)}(u) = \int_{m}^{M} h_\nu(u, p) \, dp.
\end{equation}

Our first result, Proposition \ref{basic}, expresses $h_\nu(u, p)$ in terms of the ambient support function $h_K$:
\begin{align}\label{introo1}
h_\nu(u,p)
=
\inf_{\lambda\in\mathbb R}
\{h_K(u+\lambda\nu)-\lambda p\}.
\end{align}
Combining \eqref{introo1} with \eqref{eq:fiber-body-slicing} yields
\begin{equation}\label{link2}
h_{\Sigma_{\pi_\nu}(K)}(u) = \int_{m}^{M} \inf_{\lambda \in \mathbb{R}} \bigl\{ h_K(u + \lambda\nu) - \lambda p \bigr\} \, dp.
\end{equation}
Equation~\eqref{eq:fiber-body-slicing} shows that the slicing support function is a fiberwise refinement of the support function of the fiber body; equation~\eqref{link2} further relates the support function of the fiber body directly to the support function of the original convex body $K$.

In addition, in Proposition \ref{minimax2} we show that a minimax identity for $h_\nu$ can be derived (as shown in \eqref{infsup})
\begin{align*}
h_\nu(u,p)
=
\sup_{x \in K}
\inf_{\lambda \in \mathbb{R}} L(x,\lambda)
=
\inf_{\lambda \in \mathbb{R}}
\sup_{x \in K} L(x,\lambda),
\end{align*}
where $L(x,\lambda)
:=
x \cdot u + \lambda (x \cdot \nu - p)$.
The function $L(x, \lambda)$ has a geometric meaning: the term $x\cdot u$ measures the height of $x$ along the direction $u$ of the slice, while $x\cdot\nu-p$ measures the signed deviation from the slicing hyperplane $\{x\cdot\nu=p\}$. Thus, $\lambda$ acts as a penalty that forces $x$ onto $\{x\cdot \nu=p\}$ during optimization.

Since $h_\nu$ can be derived from the original support function $h_K$ of the convex body $K$ by \eqref{introo1}, a natural converse question is whether $h_K$ can be recovered from the slicing support functions on each slice. We show the answer is positive in Theorem \ref{thm:recovery} without any assumption of boundary regularity of $K$. More precisely, for each fixed \(u\), the representation \eqref{introo1} identifies \(h_\nu(u,\cdot)\) with the negative Legendre-Fenchel transform of the one-variable convex function
\[
H_u(\lambda)=h_K(u+\lambda\nu);
\]
(see equation \eqref{temp2}). Consequently, the Fenchel-Moreau theorem yields the recovery formula
\begin{align}\label{recov1}
h_K(u+\lambda\nu)
=
\sup_{p\in\mathbb R}\{h_\nu(u,p)+\lambda p\}
\end{align}
as shown in \eqref{eq:recovery}. Thus, the relation between \(h_K\) and \(h_\nu\) is governed by Legendre--Fenchel duality and the Fenchel-Moreau biconjugation theorem; see Rockafellar~\cite{Rockafellar} and Z\u{a}linescu~\cite{zalinescu2002} for detailed accounts of these theorems.

Since \(h_K\) determines \(K\) through
\[
K=\bigcap_{u\in\mathbb{S}^{n-1}}
\{x\in\mathbb{R}^n:x\cdot u\leq h_K(u)\},
\]
the recovery formula \eqref{recov1} has two applications. First, the inclusion of one convex body in another can be characterized by comparing their slicing support functions; see Corollary~\ref{thm:order} and Corollary~\ref{cor:lossless}. Second, it yields a relation between slicing support functions and partial supremal convolution, analogous to the classical relation between support functions and Minkowski addition; see Corollary ~\ref{thm:minkowski}.

% Fix \(\nu\in \mathbb S^{n-1}\) and let $\nu^\perp=\{x\in \mathbb{R}^n,x\cdot \nu=0 \}$ be the orthogonal complement of $\nu$ in $\mathbb{R}^n$. For a slice height \(p\in\mathbb R\), consider the hyperplane section
% \[
% C(\nu,p)=K\cap\{x\in\mathbb R^n:x\cdot\nu=p\}.
% \]
% For any vector \(u\in \nu^\perp\), we define the slicing support function by
%\begin{align}\label{intro0}
%h_\nu(u,p)=\sup_{x\in C(\nu,p)}x\cdot u.
%\end{align}
% Thus, instead of assigning to each slice only a scalar quantity such as volume, \(h_\nu\) records the support function of the slice itself and in this paper we will study the properties of $h_\nu$.

% As the first result, we derive the exact representation for the slicing support function in terms of the original support function $h_K$, namely,
% \begin{align}\label{introo1}
% h_\nu(u,p)
% =
% \inf_{\lambda\in\mathbb R}
% \{h_K(u+\lambda\nu)-\lambda p\}
% \end{align}

One also obtains a differential-type recovery formula when the boundary $\partial K$ has better regularity and $K$ satisfies an appropriate convexity assumption. In Theorem \ref{thm:recovery-strict}, we assume that $\partial K$ is of class $C^1$ and that $K$ is strictly convex. Then, for each fixed $u\in \nu^\perp$ and $p\in (m, M)$, the slicing support function $h_\nu$ derived in \eqref{introo1} attains the unique minimizer $\lambda=\lambda(u,p)$ satisfying $\partial_p h_\nu(u,p)=-\lambda(u,p)$, and  the support function $h_K$ can be recovered from $h_\nu$ by the formula
\begin{align}\label{recov2}
h_K\bigl(u-\partial_p h_\nu(u,p)\,\nu\bigr)
=
h_\nu(u,p)-p\,\partial_p h_\nu(u,p).
\end{align}

If $K$ is strictly conves in $\mathbb{R}^3$ with $C^2$-boundary and positive Gaussian curvature, we derive a Monge--Amp\`ere-type identity. Let $e_1,e_2$ be an orthonormal basis of $\nu^\perp$, and write the radial direction in the slice as
$u(\theta)=(\cos\theta)e_1+(\sin\theta)e_2.$ Then the slicing support function can be viewed as
\[
H(\theta,p):=h_\nu(u(\theta),p)
\]
on the cylinder $\mathbb{S}^1\times I_\nu$, where $I_\nu$ denotes the interior of the projection interval of $K$ in the $\nu$-direction. We consider
\begin{align}\label{determ1}
\mathcal A[H]
:=
\det
\begin{pmatrix}
H_{\theta\theta}+H & H_{\theta p}\\
H_{\theta p} & H_{pp}
\end{pmatrix},
\end{align}
where subscripts denote partial derivatives with respect to the corresponding variables.

On the ambient side, write
$F(\theta,\lambda):=h_K(u(\theta)+\lambda\nu)$, and denote the corresponding minimizing parameter by
$\lambda=\lambda(\theta,p)$. With this notation, Theorem~\ref{cylinder} gives
\begin{align}\label{intro2}
\mathcal A[H]
=
-\frac{F_{\theta\theta}+F-\lambda F_\lambda}{F_{\lambda\lambda}},
\end{align}
where the right-hand side is evaluated at the minimizer. Furthermore,
Corollary~\ref{prop:second-derivation-rhs} expresses \eqref{intro2} in terms
of the spherical curvature matrix $Q_{\bar h_K}$ defined in
\eqref{sphericalcurvature1}. More precisely,
\begin{align}\label{operatorA}
\mathcal A[H]
=
-(1+\lambda^2)\,
\frac{Q_\omega[u^\perp,u^\perp]}
     {Q_\omega[\tau,\tau]},
\end{align}
where $Q_\omega:=Q_{\bar h_K}(\omega)$ is the spherical curvature matrix at
\[
\omega=\frac{u(\theta)+\lambda\nu}
{|u(\theta)+\lambda\nu|},
\]
and $u^\perp,\tau$ are the corresponding orthogonal tangent vectors of
$\mathbb{S}^2$ at $\omega$; the detailed notation is given in
Section~\ref{sec4}.

In particular, if $u^\perp$ and $\tau$ are principal directions at
$\omega$, then \eqref{operatorA} reduces to a weighted ratio of the
corresponding principal curvatures; see
\eqref{eq:principal-curvature-ratio}. Thus, the second-order structure
$\mathcal A[H]$ directly reflects the local curvature geometry of
$\partial K$.

Moreover, Proposition~\ref{prop:principal-revolution} characterizes when
$u^\perp$ and $\tau$ are principal directions. More precisely, if
$K\subset\mathbb{R}^3$ is a convex body with $C^2$-boundary and positive
Gaussian curvature, then $u^\perp$ and $\tau$ are principal directions at
every $\omega\in\mathbb{S}^2\setminus\{\pm\nu\}$ if and only if, up to
translation, $K$ is a body of revolution about an axis parallel to $\nu$.
Consequently, for such a body of revolution, $\mathcal A[H]$ is everywhere
given by a weighted ratio of the corresponding principal curvatures.

As a final result, we introduce successive slicing support functions. These are obtained by iterating the slicing procedure within sections of lower dimension of $K\subset \mathbb R^n$. First, fix a unit vector $\nu_1\in \mathbb{S}^{n-1}$ and a height $p_1$, and consider the affine hyperplane
$H_{\nu_1,p_1}=\{x\in \mathbb{R}^n:x\cdot \nu_1=p_1\}$.
On the section $K\cap H_{\nu_1,p_1}$, the slicing support function $h_{\nu_1}^{(1)}$ is defined by \eqref{intro0}, with $C(\nu, p)$ replaced by $K \cap H_{\nu_1, p_1}$. Next, choose a unit vector $\nu_2\in \nu_1^\perp$, so that $\nu_2$ gives a direction in the hyperplane $H_{\nu_1,p_1}$. Since every section of a convex body is again convex in the corresponding lower-dimensional affine subspace, we may slice $K\cap H_{\nu_1,p_1}$ once more by the affine hyperplane
$H_{\nu_2,p_2}=\{x\in \mathbb{R}^n:x\cdot \nu_2=p_2\}$.
The resulting section $K\cap H_{\nu_1,p_1}\cap H_{\nu_2,p_2}$
is again convex, now lying in an affine subspace of codimension two. Hence one can define a second slicing support function, denoted by
$h_{(\nu_1,\nu_2)}^{(2)}$ on this new section. Continuing this procedure, we obtain \textit{the $k$-th slicing support function}
\begin{align}\label{iteradef}
h_{\boldsymbol{\nu}_k}^{(k)}(u, \mathbf{p}_k) = \sup_{x \in K^{(k)}(\boldsymbol{\nu}_k, \mathbf{p}_k)} x \cdot u, \quad \text{for } u \in E_{n-k},
\end{align}
where
$\nu_1,\cdots, \nu_k\in \mathbb{S}^{n-1}$ are orthogonal unit vectors, $\boldsymbol{\nu}_k=(\nu_1,\cdots, \nu_k)$, $\mathbf{p}_k=(p_1, \cdots, p_k)\in \mathbb{R}^k$,
$E_{n-k}$ is the orthogonal complement of the space $V_k=span\{\nu_1,\cdots, \nu_k \}$ in $\mathbb{R}^n$, and
\begin{align}\label{ktimesslide}
K^{(k)}(\boldsymbol{\nu}_k, \mathbf{p}_k) = K \bigcap_{j=1}^k \{x \in \mathbb{R}^n : x \cdot \nu_j = p_j\}.
\end{align}
Note that when $k=1$, the slicing support function defined in \eqref{iteradef} coincides with the one defined in \eqref{intro0}. Further details are given in the paragraph preceding Definition \ref{defiterate}.

Although the definition in \eqref{iteradef} is constructed by iterating the geometric slicing procedure, it admits a more useful representation. Theorem \ref{iteratedrepre} expresses the $k$-th slicing support function in terms of the ambient support function $h_K$ as
\begin{equation} \label{eq:global_inf}
h_{{\boldsymbol{\nu}}_k}^{(k)}(u, \mathbf{p}_k) = \inf_{\boldsymbol{\lambda} \in \mathbb{R}^k} \left( h_K\left(u + \sum_{j=1}^k \lambda_j \nu_j\right) - \sum_{j=1}^k \lambda_j p_j \right).
\end{equation}
Thus, the $k$-th slicing support function has the same dual representation as the first slicing support function in \eqref{introo1}, with the real-valued parameter $\lambda$ replaced by the vector-valued parameter $\boldsymbol{\lambda}=(\lambda_1,\ldots,\lambda_k)$.

One application of the representation formula for the $k$-th slicing support function in \eqref{eq:global_inf} is that it reveals an analogous Monge-Amp\`ere structure in the higher-codimensional setting. In Theorem \ref{fullhessian}, we use this representation to derive a determinant identity for the full Hessian of $h_{\boldsymbol{\nu}_k}^{(k)}$, extending the three-dimensional identity \eqref{intro2} to higher-dimensional spaces.

%this is the setup section
\section{Definition and Basic Properties of Slicing Support Functions}\label{sec2}

Let $K \subset \mathbb{R}^n$ be a convex body.
For any unit vector $\nu \in \mathbb{S}^{n-1}$ and any $p \in \mathbb{R}$, we define the slicing hyperplane orthogonal to $\nu$ at height $p$ by
\[
H_{\nu, p} := \{ x \in \mathbb{R}^n \; : \; x \cdot \nu = p \},
\]
where $x\cdot \nu$ represents the usual Euclidean inner product. The corresponding slice (hyperplane section) of $K$ by $H_{\nu,p}$ is defined as
\[
C(\nu,p) := K \cap H_{\nu, p};
\]
it is understood that $C(\nu,p)$ may be empty. Denote by $\nu^\perp$ the subspace perpendicular to $\nu$. Let $u \in \nu^\perp$, the direction space of the affine hyperplane $H_{\nu,p}$; equivalently, $u \cdot \nu = 0.$
The \emph{slicing support function} of $K$ on the slice $C(\nu,p)$ in the direction $u$ is defined by
\begin{equation}\label{def1}
h_\nu(u,p)
=
\sup_{x \in C(\nu,p)} x \cdot u
\end{equation}
with the convention $h_\nu (u,p)=-\infty$ if $C(\nu,p)=\emptyset$. We have the infimal representation of $h_\nu$ from $h_K$.

\begin{proposition}\label{basic}
For any fixed \(p\in\mathbb{R}\) and any \(u\in\nu^\perp\), the slicing
support function satisfies
\begin{equation}\label{infhk}
h_\nu(u,p)
=
\inf_{\lambda\in\mathbb{R}}
\bigl\{h_K(u+\lambda\nu)-\lambda p\bigr\},
\end{equation}
where \(h_K\) denotes the support function of \(K\).
\end{proposition}

\begin{proof}
We first prove
\begin{align}\label{ineq:first}
h_\nu(u,p)
\le
\inf_{\lambda\in\mathbb{R}}
\bigl\{h_K(u+\lambda\nu)-\lambda p\bigr\}.
\end{align}
If \(C(\nu,p)=\emptyset \), then \(h_\nu(u,p)=-\infty\), and the inequality is immediate. Suppose that \(C(\nu,p)\neq\emptyset\). For any $x\in C(\nu,p)$ and any $\lambda \in \mathbb{R}$, we have
\begin{align}\label{leq1}
x\cdot u
&= x\cdot u + \lambda(x\cdot \nu - p) \nonumber \\
&= x\cdot (u+\lambda \nu) - \lambda p.
\end{align}
Since $x\in K$, it follows that
\begin{align}\label{leq2}
x\cdot (u+\lambda \nu) \le \sup_{y\in K} y\cdot (u+\lambda \nu) = h_K(u+\lambda \nu).
\end{align}
Therefore, \eqref{leq1} and \eqref{leq2} imply
$x\cdot u \le h_K(u+\lambda \nu) - \lambda p$. Taking the supremum over all $x\in C(\nu,p)$ gives \[ h_\nu(u,p) \le h_K(u+\lambda \nu)-\lambda p. \] Since this holds for every $\lambda\in\mathbb{R}$, we obtain \[ h_\nu(u,p) \le \inf_{\lambda\in\mathbb{R}} \bigl(h_K(u+\lambda \nu)-\lambda p\bigr), \] which proves \eqref{ineq:first}.

% For every \(x\in C(\nu,p)\) and every \(\lambda\in\mathbb{R}\), we have
% \[
% x\cdot u
% =
% x\cdot u+\lambda(x\cdot\nu-p)
% =
% x\cdot(u+\lambda\nu)-\lambda p.
% \]
% Since $x\in K$, $x\cdot(u+\lambda\nu)
% \le h_K(u+\lambda\nu)$.
% Therefore,
% $x\cdot u
% \le h_K(u+\lambda\nu)-\lambda p$.
% Taking the supremum over \(x\in C(\nu,p)\), and then the infimum over
% \(\lambda\in\mathbb{R}\), gives
% \[
% h_\nu(u,p)
% \le
% \inf_{\lambda\in\mathbb{R}}
% \bigl\{h_K(u+\lambda\nu)-\lambda p\bigr\}.
% \]

We now prove the reverse inequality. Set
\[
m:=\min_{x\in K}x\cdot\nu=-h_K(-\nu),
\qquad
M:=\max_{x\in K}x\cdot\nu=h_K(\nu),
\]
and define
\[
f(q):=
\sup\{x\cdot u:x\in K,\ x\cdot\nu=q\},
\qquad q\in\mathbb{R},
\]
with the convention \(f(q)=-\infty\) if the corresponding slice is
empty. Since the image of \(K\) under the map \(x\mapsto x\cdot\nu\)
is the interval \([m,M]\), we have
\[
f(q)=h_\nu(u,q)\in\mathbb{R}
\quad\text{for }q\in[m,M],
\]
and \(f(q)=-\infty\) for \(q\notin[m,M]\).

Now we claim that \( f \) is concave on \([m, M]\) due to the convexity of \(K\). Indeed, let \(q_1, q_2 \in [m, M]\) and \(t \in [0,1]\). By the definition of \(f\), for any \(\varepsilon > 0\), we can choose points \(x_i \in K\) such that $
x_i \cdot \nu = q_i$ and $x_i \cdot u \ge f(q_i) - \varepsilon,  (i = 1,2).$
By the convexity of $K$, the point $x_t := t x_1 + (1 - t)x_2\in K$; by linearity, we have
\begin{align*}
x_t \cdot \nu &= t q_1 + (1 - t) q_2 \in [m, M], \\
x_t \cdot u &= t(x_1 \cdot u) + (1 - t)(x_2 \cdot u).
\end{align*}
Thus, \(x_t\) is admissible for \(f(t q_1 + (1 - t) q_2)\), and hence
\[
f(t q_1 + (1 - t) q_2) \ge x_t \cdot u
\ge t f(q_1) + (1 - t) f(q_2) - \varepsilon.
\]
Letting \(\varepsilon \downarrow 0\), we obtain
\[
f(t q_1 + (1 - t) q_2) \ge t f(q_1) + (1 - t) f(q_2),
\]
which shows that \(f\) is concave on \([m, M]\).

We shall also use the identity: for all $\alpha\in \mathbb{R}$,
\begin{align}\label{eq:fiber-envelope}
h_K(u+\alpha \nu)
&=\sup_{x\in K} x\cdot (u+\alpha \nu) =\sup_{x\in K} \bigl(x\cdot u+\alpha\, x\cdot \nu\bigr) \\
&=\sup_{q\in \mathbb{R}} \sup_{\substack{x\in K\\ x\cdot \nu=q}}
\bigl(x\cdot u+\alpha q\bigr)=\sup_{q\in \mathbb{R}} \bigl(f(q)+\alpha q\bigr).\nonumber
\end{align}

% We shall also use the identity
% \begin{equation}\label{eq:fiber-envelope}
% h_K(u+\alpha\nu)
% =
% \sup_{q\in[m,M]}
% \bigl\{f(q)+\alpha q\bigr\},
% \qquad \alpha\in\mathbb{R}.
% \end{equation}
% Indeed,
% \begin{align*}
% h_K(u+\alpha\nu)
% &=
% \sup_{x\in K}
% \bigl\{x\cdot u+\alpha x\cdot\nu\bigr\} \\
% &=
% \sup_{q\in[m,M]}
% \sup_{\substack{x\in K\\x\cdot\nu=q}}
% \bigl\{x\cdot u+\alpha q\bigr\} \\
% &=
% \sup_{q\in[m,M]}
% \bigl\{f(q)+\alpha q\bigr\}.
% \end{align*}

Now we consider four cases.
\emph{Case 1: \(p\in(m,M)\).}
Since \(f\) is finite and concave on \([m,M]\), it admits a finite
supporting slope \(s\in\mathbb{R}\) at \(p\); that is,
\[
f(q)\le f(p)+s(q-p)
\qquad\text{for every }q\in[m,M].
\]
For completeness, one may choose \(s\) satisfying
\[
\sup_{q\in(p,M]}
\frac{f(q)-f(p)}{q-p}
\le s\le
\inf_{q\in[m,p)}
\frac{f(q)-f(p)}{q-p}.
\]
The two quantities are finite, and the displayed inequality follows
from the concavity of \(f\).

Using \eqref{eq:fiber-envelope} with \(\alpha=-s\), we obtain
\begin{align*}
h_K(u-s\nu)+sp
&=
\sup_{q\in[m,M]}
\bigl\{f(q)-sq+sp\bigr\} \\
&=
\sup_{q\in[m,M]}
\bigl\{f(q)-s(q-p)\bigr\} \\
&\le f(p).
\end{align*}
Therefore,
\[
\inf_{\lambda\in\mathbb{R}}
\bigl\{h_K(u+\lambda\nu)-\lambda p\bigr\}
\le f(p)
=h_\nu(u,p).
\]

\medskip
\noindent
\emph{Case 2: \(p=M\).}
For \(t>0\), define
\[
G_M(t):=h_K(u+t\nu)-tM.
\]
By \eqref{eq:fiber-envelope},
\[
G_M(t)
=
\max_{x\in K}
\bigl\{x\cdot u-t(M-x\cdot\nu)\bigr\}.
\]
Since \(M-x\cdot\nu\ge0\) for every \(x\in K\), the function \(G_M\)
is nonincreasing. Moreover, by evaluating the maximum on the top
slice \(C(\nu,M)\), we obtain
$G_M(t)\ge f(M)$. Hence the finite limit
\[
L_M:=\lim_{t\to\infty}G_M(t)
\]
exists and satisfies \(L_M\ge f(M)\).

For each \(t>0\), choose \(x_t\in K\) such that
\[
G_M(t)
=
x_t\cdot u-t(M-x_t\cdot\nu).
\]
Let
\[
C_u:=\max_{x\in K}x\cdot u.
\]
Since \(G_M(t)\ge f(M)\), we have
\[
f(M)
\le x_t\cdot u-t(M-x_t\cdot\nu)
\le C_u-t(M-x_t\cdot\nu).
\]
It follows that
\[
0\le M-x_t\cdot\nu
\le\frac{C_u-f(M)}{t},
\]
and hence $x_t\cdot\nu\longrightarrow M$ as $t\to\infty$. Choose a sequence \(t_j\to\infty\). By the compactness of \(K\), after
passing to a subsequence, we may assume that
$x_{t_j}\longrightarrow x_\infty\in K$.
Then \(x_\infty\cdot\nu=M\), so \(x_\infty\in C(\nu,M)\). Furthermore,
\[
G_M(t_j)
=
x_{t_j}\cdot u-t_j(M-x_{t_j}\cdot\nu)
\le x_{t_j}\cdot u.
\]
Passing to the limit gives
\[
L_M\le x_\infty\cdot u\le f(M).
\]
Thus \(L_M=f(M)\), and therefore
\[
\inf_{\lambda\in\mathbb{R}}
\bigl\{h_K(u+\lambda\nu)-\lambda M\bigr\}
\le
\lim_{t\to\infty}
\bigl\{h_K(u+t\nu)-tM\bigr\}
=f(M).
\]

\medskip
\noindent
\emph{Case 3: \(p=m\).}
For \(t>0\), define
\[
G_m(t):=h_K(u-t\nu)+tm.
\]
As above,
\[
G_m(t)
=
\max_{x\in K}
\bigl\{x\cdot u-t(x\cdot\nu-m)\bigr\}.
\]
Since \(x\cdot\nu-m\ge0\) for every \(x\in K\), the function \(G_m\)
is nonincreasing and satisfies
$G_m(t)\ge f(m)$.
For each \(t>0\), choose \(y_t\in K\) attaining this maximum. Then
\[
f(m)
\le y_t\cdot u-t(y_t\cdot\nu-m)
\le C_u-t(y_t\cdot\nu-m),
\]
and hence
\[
0\le y_t\cdot\nu-m
\le\frac{C_u-f(m)}{t}.
\]
Thus $y_t\cdot\nu\to m$ as $t\rightarrow\infty$. Using the compactness of \(K\) and arguing
as in Case 2, we obtain
$\lim_{t\to\infty}G_m(t)=f(m)$.
Consequently,
\[
\inf_{\lambda\in\mathbb{R}}
\bigl\{h_K(u+\lambda\nu)-\lambda m\bigr\}
\le
\lim_{t\to\infty}
\bigl\{h_K(u-t\nu)+tm\bigr\}
=f(m).
\]

\medskip
\noindent
\emph{Case 4: \(p\notin[m,M]\).}
If \(p>M\), then for \(t>0\), the subadditivity and positive
homogeneity of \(h_K\) give
\begin{align*}
h_K(u+t\nu)-tp
&\le h_K(u)+t h_K(\nu)-tp \\
&=h_K(u)-t(p-M)
\longrightarrow-\infty
\end{align*}
as \(t\to\infty\). Therefore,
\[
\inf_{\lambda\in\mathbb{R}}
\bigl\{h_K(u+\lambda\nu)-\lambda p\bigr\}
=-\infty.
\]

If \(p<m\), then
\begin{align*}
h_K(u-t\nu)+tp
&\le h_K(u)+t h_K(-\nu)+tp \\
&=h_K(u)-t(m-p)
\longrightarrow-\infty
\end{align*}
as \(t\to\infty\). Thus the infimum is again \(-\infty\).
Since \(C(\nu,p)=\varnothing\) in either case, this agrees with
\(h_\nu(u,p)=-\infty\).

Combining the four cases gives
\[
\inf_{\lambda\in\mathbb{R}}
\bigl\{h_K(u+\lambda\nu)-\lambda p\bigr\}
\le h_\nu(u,p).
\]
Together with the first inequality, this proves
\eqref{infhk}.
\end{proof}

Next, we show that the slicing support function admits the following minimax representation.

\begin{proposition}\label{minimax2}
Let $K\subset \mathbb{R}^n$ be a convex body. For any $p\in \mathbb{R}$ and any $u\in \nu^\perp$, define the function
\[
L(x,\lambda)
:=
x \cdot u + \lambda (x \cdot \nu - p),
\qquad
x \in K, \; \lambda \in \mathbb{R}.
\]
Then the slicing support function defined by \eqref{def1} can be expressed as
\begin{equation}\label{infsup}
h_\nu(u,p)
=
\sup_{x \in K}
\inf_{\lambda \in \mathbb{R}} L(x,\lambda)
=
\inf_{\lambda \in \mathbb{R}}
\sup_{x \in K} L(x,\lambda).
\end{equation}
\end{proposition}

\begin{proof}
First, fix $x\in K$ and consider
$$
\inf_{\lambda\in\mathbb{R}} L(x,\lambda)
=
\inf_{\lambda\in \mathbb{R}}
\bigl(x\cdot u + \lambda (x\cdot \nu - p)\bigr).
$$
If $x\cdot \nu \neq p$, then the term $\lambda(x\cdot \nu - p)$
is unbounded below as $\lambda \to \pm\infty$, depending on the sign of $x\cdot \nu - p$.
Hence
\[
\inf_{\lambda\in \mathbb{R}} L(x,\lambda) = -\infty,
\qquad
\text{whenever } x\cdot \nu \neq p;
\]
if instead $x$ lies in the slicing plane $C(\nu,p)$, then the $\lambda$–term vanishes and
\[
\inf_{\lambda} L(x,\lambda)
=
x\cdot u.
\]
Combining the two cases, we obtain
\begin{align*}
\sup_{x\in K}\inf_{\lambda} L(x,\lambda)
&=\sup_{x\in K}\Bigl(\{x\cdot u : x\in C(\nu,p)\}\cup\{-\infty\}\Bigr)\\
&=
\sup_{x\in C(\nu,p)} x\cdot u
=
h_\nu(u,p).
\end{align*}

On the other hand, using \eqref{infhk} and a straightforward computation, we obtain
\begin{align*}
\inf_{\lambda} \sup_{x\in K} L(x,\lambda)
&=
\inf_{\lambda}
\left(
\sup_{x\in K} x\cdot (u+\lambda \nu) - \lambda p
\right) \\
&=
\inf_{\lambda}
\bigl(
h_K(u+\lambda \nu) - \lambda p
\bigr) \\
&=
h_\nu(u,p),
\end{align*}
and the result follows.
\end{proof}

\section{Recovering the Support Function from Slicing Support Functions}\label{sec3}
We will show that the support function $h_K$ of $K$ can be recovered from the slicing support function $h_\nu$, both without and with regularity assumptions on $\partial K$. We first recover $h_K$ from $h_\nu$ without any regularity assumption on $\partial K$. Since the proof relies on biconjugation, we begin with the following auxiliary lemma.

\begin{lemma}\label{lem:fu-properties}
Let $K\subset \mathbb{R}^n$ be a convex body. For every fixed $u\in \nu^\perp$, the function
\[
H_u(\lambda):=h_K(u+\lambda \nu), \qquad \lambda\in\mathbb{R},
\]
is proper, convex, and lower semicontinuous on $\mathbb{R}$.
\end{lemma}

\begin{proof}
First, since $K$ is compact, $h_K(v)=\sup_{x\in K} x\cdot v$ is finite for every $v\in\mathbb{R}^n$. Hence $H_u(\lambda)\in\mathbb{R}$ for all $\lambda$, so $H_u$ is proper.

Secondly, the support function $h_K$ is convex on $\mathbb{R}^n$ since it is a supremum of linear functions, $h_K(v)=\sup_{x\in K} x\cdot v.$
Fix $\lambda_0,\lambda_1\in\mathbb{R}$ and $t\in[0,1]$. Since
\[
u+\bigl(t\lambda_0+(1-t)\lambda_1\bigr)\nu
=
t(u+\lambda_0 \nu)+(1-t)(u+\lambda_1 \nu)
\]
and $h_K$ is convex, one gets
\begin{align*}
H_u\bigl(t\lambda_0+(1-t)\lambda_1\bigr)
&=
h_K\!\left(t(u+\lambda_0 \nu)+(1-t)(u+\lambda_1 \nu)\right)\\
&\le
t\,h_K(u+\lambda_0 \nu)+(1-t)\,h_K(u+\lambda_1 \nu) \\
&=
t H_u(\lambda_0)+(1-t)H_u(\lambda_1).
\end{align*}
Thus, $H_u$ is convex on $\mathbb{R}$.

Finally, since $K$ is compact, $h_K$ is finite and continuous on $\mathbb{R}^n$. Therefore the composition $H_u(\lambda)=h_K(u+\lambda \nu)$ is continuous on $\mathbb{R}$, hence lower semicontinuous.
\end{proof}

For our purposes, let us first recall the Fenchel--Moreau theorem.
For further details, we refer to \cite{koshi1983,zalinescu2002} and \cite[Section 8]{simons2008hahn}. Suppose that $X$ is a real Hausdorff locally convex topological vector space,
and let $X^{*}$ denote its dual space. We write
$\langle \cdot,\cdot \rangle : X^{*} \times X \to \mathbb{R}$
for the canonical dual pairing, defined by
\[
\langle x^{*},x\rangle := x^{*}(x),
\qquad x^{*}\in X^{*}, \; x\in X.
\]
Let $f:X\to (-\infty,+\infty]$ be an extended real-valued function.
The \emph{convex conjugate} (or the Legendre--Fenchel transform)
of $f$ is the function $f^{*}:X^{*}\to (-\infty,+\infty]$ defined by
\[
f^{*}(x^{*})
:=
\sup_{x\in X}
\bigl\{
\langle x^{*},x\rangle - f(x)
\bigr\},
\qquad x^{*}\in X^{*}.
\]
The \emph{biconjugate} of $f$ is the function
$f^{**}:X\to (-\infty,+\infty]$ defined by
\[
f^{**}(x)
:=
\sup_{x^{*}\in X^{*}}
\bigl\{
\langle x^{*},x\rangle - f^{*}(x^{*})
\bigr\},
\qquad x\in X.
\]
Equivalently, $f^{**}$ is the convex conjugate of $f^*$, regarded again as a function on $X$. The Fenchel--Moreau theorem asserts that if $f$ is proper, convex, and lower
semicontinuous, then
\[
f^{**}=f.
\]

We will apply the Fenchel-Moreau theorem to recover the support function
from the slicing support function by taking
$X=\mathbb{R}=X^{*}$.
Under this identification, the dual pairing $\langle p,\lambda\rangle = p\lambda$ is simply the usual multiplication.

\begin{theorem}[Recovery of $h_K$ from $h_\nu$]\label{thm:recovery}
Let $K\subset\mathbb{R}^n$ be a convex body, fix $\nu \in \mathbb{S}^{n-1}$, and define the slicing support function $h_\nu(u,p)$ as above. Then for every $u\in \nu^\perp$ and $\lambda\in\mathbb{R}$,
\begin{equation}\label{eq:recovery}
h_K(u+\lambda \nu)=\sup_{p\in\mathbb{R}}\bigl(h_\nu(u,p)+\lambda p\bigr).
\end{equation}
Consequently, knowing $h_\nu(u,p)$ for all $(u,p)\in \nu^\perp\times\mathbb{R}$ determines $h_K(v)$ for all $v\in\mathbb{R}^n$. More precisely, for all $v\in \mathbb{R}^n$, we have
\begin{align}\label{recover2}
h_K(v)=\sup_{p\in\R}\Big(h_\nu\big(v-(v\cdot \nu)\ \nu, p\big)+(v\cdot \nu)\,p\Big).
\end{align}
\end{theorem}

\begin{proof}
Fix $u$ and consider $H_u$ defined as in Lemma \ref{lem:fu-properties}. Since $H_u$ is proper, convex, and lower semicontinuous, the Fenchel-Moreau theorem applies and yields
\begin{align}\label{temp1}
H_u(\lambda)=H_u^{**}(\lambda)=\sup_{p\in\mathbb{R}}\bigl(\lambda p-H_u^*(p)\bigr),\qquad \lambda\in\mathbb{R}.
\end{align}
Using \eqref{infhk}, we derive
\begin{align}\label{temp2}
    -H_u^*(p)&=-\sup_{\lambda\in\mathbb{R}}(-H_u(\lambda)+\lambda p )=\inf_{\lambda\in\mathbb{R}}(H_u(\lambda)-\lambda p)=h_\nu(u,p).
\end{align}
Substituting \eqref{temp2} into \eqref{temp1}, we obtain \eqref{eq:recovery}.

Finally, since every $v\in\mathbb{R}^n$ admits the unique decomposition
\[
v=u+\lambda\nu,\qquad u\in\nu^\perp,\quad \lambda=v\cdot\nu,
\]
formula \eqref{eq:recovery} recovers $h_K(v)$ for all $v$, and hence \eqref{recover2} follows.
\end{proof}

We have two immediate consequences from Theorem \ref{thm:recovery}.

\begin{corollary}[Order preservation and set inclusion]\label{thm:order}
Let $K,L$ be convex bodies in $\mathbb{R}^n$, and $h_{\nu,K}$ and $h_{\nu, L}$ denote the slicing support functions of $K$ and $L$, respectively, along the $\nu$-direction. If
\[
h_{\nu, K}(u,p)\leq h_{\nu, L}(u,p),
\]
for all $u\in \nu^\perp$ and any $p\in\R$,
then
\[
h_{K}(v)\le h_{L}(v)\qquad \forall\,v\in\R^n,
\]
and consequently $K\subset L$.
\end{corollary}

\begin{proof}
Using the recovery formula \eqref{eq:recovery}, for any fixed $u,\lambda$, we have
\begin{align*}
h_{K}(u+\lambda \nu)&=\sup_p(h_{\nu, K}(u,p)+\lambda p)\\
&\le \sup_p(h_{\nu, L}(u,p)+\lambda p)\\
&=h_{L}(u+\lambda \nu),
\end{align*}
and then extend to all $v=u+\lambda \nu$ by decomposition. The characterization $K\subset L$ if and only if $h_K\leq h_L$ for convex bodies then yields the result.
\end{proof}

\begin{corollary}[Slicing support functions uniquely determine convex bodies]\label{cor:lossless}
If the convex bodies $K$ and $L$ in $\mathbb{R}^n$ have the same values of their slicing support functions in the fixed slicing direction $\nu$,
$$h_{\nu, K}(u,p)=h_{\nu, L}(u,p)$$
for all $u\in \nu^\perp$ and all $p\in\R$, then $h_K\equiv h_L$ on $\R^n$, hence $K=L$.
\end{corollary}

\begin{proof}
By Theorem~\ref{thm:recovery}, equality of the slicing support function implies equality of the support functions on all vectors of the form $u+\lambda \nu$ with $u\in \nu ^\perp$, $\lambda\in\R$.
Every $v\in\R^n$ can be decomposed uniquely as $v=u+\lambda \nu$ with $u=v-(v\cdot \nu)\nu \in \nu^\perp$ and $\lambda=v\cdot \nu$, so $h_K(v)=h_L(v)$ for all $v$.
Finally, support functions determine compact convex sets uniquely and the result follows.
\end{proof}

Next, we will show the other application of Theorem \ref{thm:recovery}. Recall that the Minkowski sum of two subsets $K,L\subset \mathbb{R}^n$ is defined by
\[
K+L:=\{x+y:\ x\in K,\ y\in L\}.
\]
If $K$ and $L$ are convex bodies, then their support functions satisfy
$
h_{K+L}=h_K+h_L$.

Motivated by the partial infimal convolution for functions defined on a product
space
\cite[Theorem~4.2]{simons2005fenchel},
\cite[Definition~5.2]{bauschke2008monotone}, and \cite[Definition~16.3]{simons2008hahn},
we define the partial supremal convolution in the $p$-variable by
\[
\bigl(h_{\nu,K}\overline{\square}_p h_{\nu,L}\bigr)(u,p)
:=
\sup_{\substack{p_1+p_2=p}}
\Bigl\{h_{\nu,K}(u,p_1)+h_{\nu,L}(u,p_2)\Bigr\}.
\]
For each fixed $u$, this is the usual supremal convolution of
$h_{\nu,K}(u,\cdot)$ and $h_{\nu,L}(u,\cdot)$; equivalently, it is the sign-dual of the
partial infimal convolution in the second variable.

The following result describes the relation between the Minkowski sum and the supremal convolution for slicing support functions.
\begin{corollary}\label{thm:minkowski}
For convex bodies $K, L$ in $\mathbb{R}^n$,
\[
h_{\nu, K+L}
=
h_{\nu, K}\ \overline{\square}_p\ h_{\nu, L}
\]
\end{corollary}

\begin{proof}
Fix $u$ and set $H_u^K(\lambda):=h_K(u+\lambda \nu)$ and similarly $H_u^L$. Then $H_u^{K+L}=H_u^K+H_u^L$.

Since \(H_u^K\) and \(H_u^L\) are finite continuous convex functions on \(\mathbb R\), the standard conjugation formula gives
\[
(H_u^K+H_u^L)^*
=
(H_u^K)^*\square_p (H_u^L)^*
\]
without taking the lower semicontinuous closure of the infimal convolution,
where $\square$ is \emph{infimal} convolution defined by
\[
(H_u^{K*}\ \square_p \ H_u^{L*})(p)=\inf_{p_1+p_2=p}\big(H_u^{K*}(p_1)+H_u^{L*}(p_2)\big).
\]

By \eqref{temp2}, $h_{\nu, K}=-H_u^{K*}$ and $h_{\nu, L}= - H_u^{L*}$, so
\begin{align*}
h_{\nu, K+L}(u,p)&=- (H_u^{K+L})^*(p)
= -\inf_{p_1+p_2=p}\big(H_u^{K*}(p_1)+H_u^{L*}(p_2)\big)\\ &=\sup_{p_1+p_2=p}\big(h_{\nu, K}(u,p_1)+h_{\nu, L}(u,p_2)\big).
\end{align*}
\end{proof}

The following result gives the other recovery formula under suitable regularity assumptions on the boundary $\partial K$. For this purpose, we need to impose a uniqueness assumption on the minimizer in \eqref{infhk}.

We first recall the minimizer condition. Fix $\nu \in {S}^{n-1}$, $u \in \nu^\perp$, and $p \in \mathbb{R}$. Define
\begin{align}\label{fup}
F_{u,p}(\lambda)
:=
h_K(u+\lambda \nu)-\lambda p,
\qquad \lambda \in \mathbb{R}.
\end{align}
A number $\lambda^* \in \mathbb{R}$ is called a \textit{minimizer} of $F_{u,p}$ if it attains the infimum, that is,
\[
F_{u,p}(\lambda^*) = \inf_{\lambda \in \mathbb{R}} F_{u,p}(\lambda).
\]

When $K$ is strictly convex, one expects uniqueness of the supporting point in a given direction, and correspondingly uniqueness of the minimizer in the relevant situation. Under these conditions, we will show that the ambient support function can be recovered from the slicing support function (see Theorem \ref{thm:recovery-strict} below).

To this end, we introduce a boundary condition for $K$. We say that $\partial K$ satisfies the \textit{$h$-boundary condition} if, for each $v \in \mathbb{S}^{n-1}$, there exists a unique point $x(v) \in \partial K$ such that $h_K(v) = x(v) \cdot v$, and the map $v \mapsto x(v)$ is a bijection from $\mathbb{S}^{n-1}$ onto $\partial K$. As an example, we point out that a strictly convex body $K$ with $C^1$-boundary satisfies the $h$-boundary condition. Another example is provided by Hadamard's theorem (see \cite{deLima2023} for instance). A compact, connected $C^2$-surface in $\mathbb{R}^3$ with positive Gaussian curvature bounds a strictly convex body, and its Gauss map is a diffeomorphism. Therefore it also satisfies the $h$-boundary condition.

\begin{remark}
Strict convexity of $K$ alone, without boundary $C^1$-regularity assumption, does not ensure that the $h$-boundary condition holds. For instance, in $\mathbb{R}^2$, consider the strictly convex body
\[
K=\{(x,y)\in \mathbb{R}^2 : |x|+y^2\le 1\},
\]
which has two non-smooth points at $(0,1)$ and $(0,-1)$. For any $a \in [-1,1]$, the line $L_a: ax+2y=2$ is a supporting line of $K$ at $(0,1)$. Indeed, for any $(x,y)\in K$, we have
$ax+2y \le |x|+2y \le |x|+y^2+1 \le 2$.
Moreover, the equality holds if and only if $(x,y)=(0,1)$, so $(0,1)$ is the unique contact point of each $L_a$. Let $v$ be a unit normal vector to $L_a$. Then $h_K(v)=x(v)\cdot v$ with $x(v)=(0,1)$. Consequently, many distinct directions $v$ correspond to the same support point, and the map $v \mapsto x(v)$ fails to be bijective. Therefore, $K$ does not satisfy the $h$-boundary condition.
\end{remark}

\begin{lemma}[Existence and uniqueness of the minimizer under $h$-boundary condition]\label{lem:unique-minimizer}
Let $K \subset \mathbb{R}^n$ be a convex body with the $h$-boundary condition. Fix $\nu \in \mathbb{S}^{n-1}$ and nonzero $u \in \nu^\perp$, and let $p \in (m,M)$, where $m$ and $M$ are defined in \eqref{height}. Then the function $F_{u,p}$ defined in \eqref{fup} admits a unique minimizer $\lambda^* = \lambda^*(u,\nu,p) \in \mathbb{R}$.
\end{lemma}

\begin{remark}
The interior condition, $m<p<M$, is essential for existence of a finite minimizer. At the endpoint heights $p=m$ or $p=M$, the infimum may fail to be attained since the planar section $C(\nu, p)$ of $K$ degenerates to a single point.
\end{remark}

\begin{proof}
We split the proof into existence and uniqueness.

\noindent
\textit{Step 1: Existence.}
Since $p<M=h_K(\nu)$, choose $\varepsilon_+>0$ such that
\[
p\le h_K(\nu)-2\varepsilon_+.
\]
Moreover, the support function is sublinear and \(1\)-homogeneous,
\[
h_K(u+\lambda \nu)\ge h_K(\lambda \nu)-h_K(-u)=\lambda h_K(\nu)-h_K(-u)
\qquad (\lambda\ge 0).
\]
Hence for $\lambda\ge 0$,
\begin{align*}
F_{u, p}(\lambda)
&=
h_K(u+\lambda \nu)-\lambda p \\
&\ge
\lambda\bigl(h_K(\nu)-p\bigr)-h_K(-u) \\
&\ge
2\lambda \varepsilon_+ -h_K(-u)\xrightarrow[\lambda\to +\infty]{}+\infty.
\end{align*}
Similarly, since $p>m=-h_K(-\nu)$, choose $\varepsilon_->0$ such that
\[
p\ge -h_K(-\nu)+2\varepsilon_-.
\]
Write $\lambda=-t$ with $t\ge 0$.
Again by sublinearity,
\[
h_K(u-t \nu)\ge h_K(-t \nu)-h_K(-u)=t\,h_K(-\nu)-h_K(-u).
\]
Therefore,
\begin{align*}
F_{u, p}(-t)
&=
h_K(u-t\nu)+tp \\
&\ge
t\bigl(h_K(-\nu)+p\bigr)-h_K(-u) \\
&\ge
2 t\varepsilon_- -h_K(-u)\xrightarrow[t\to +\infty]{}+\infty.
\end{align*}
So $F_{u, p}(\lambda)\to+\infty$ as $|\lambda|\to\infty$. Since $K$ is convex, $h_K$ (and so $F_{u, p}$) is continuous, $F_{u,p}$ must
attain its minimum on $\mathbb{R}$.

\noindent
\textit{Step 2: Uniqueness.}
For each $\lambda\in\mathbb{R}$, let $v_\lambda:=u+\lambda \nu.$
Since $u\neq 0$ and $u\perp \nu$, we have $v_\lambda\neq 0$ for all $\lambda$. Since $K$ satisfies the $h$-boundary condition, each supporting direction $v_\lambda$ determines a unique contact point
$x(\lambda)$ on the boundary of $K$ such that
$h_K(v_\lambda)=x(\lambda)\cdot v_\lambda$.
We claim that the scalar function
\[
\lambda\longmapsto x(\lambda)\cdot \nu
\]
is strictly increasing. Let $\lambda_1<\lambda_2$, and write $x_i:=x(\lambda_i)$,
$v_i:=v_{\lambda_i}$ for $i=1,2$. Since $v_1\neq v_2$, the $h$-boundary assumption gives $x_1\neq x_2$.
Since $x_i$ is the unique contact point in direction $v_i$, we have the strict support inequalities
\[
x_1\cdot v_1=h_K(v_1) >x_2\cdot v_1,
\qquad
x_2\cdot v_2 =h_K(v_2)>x_1\cdot v_2.
\]
Adding them yields
\[
(x_1-x_2)\cdot(v_1-v_2)>0.
\]
But $v_1-v_2=(\lambda_1-\lambda_2)\nu$,
hence $(\lambda_1-\lambda_2)\bigl((x_1-x_2)\cdot \nu \bigr)>0$.
Since $\lambda_1-\lambda_2<0$ by assumption, it follows that $x_1\cdot \nu <x_2\cdot \nu$.
So $\lambda\mapsto x(\lambda)\cdot \nu$ is strictly increasing and this proves the claim.

Now let $\lambda$ be any point at which $F_{u, p}$ is differentiable. Since $h_K$ is differentiable at $v_\lambda$, the definition of \(F_{u,p}\) gives
\[
F'_{u, p}(\lambda)=\nabla h_K(v_\lambda)\cdot \nu-p=x(\lambda)\cdot \nu-p.
\]
Since $\lambda\mapsto x(\lambda)\cdot \nu$ is strictly increasing, the derivative
$F'_{u, p}(\lambda)$ is strictly increasing. Thus, $F_{u, p}$ is strictly convex on $\mathbb{R}$.

A strictly convex function can have at most one minimizer. Since existence was established in Step
1, the minimizer of $F_{u,p}$ is unique.
\end{proof}

\begin{proposition}[Derivative of the slicing support function with respect to the slice height]\label{prop:lambda-derivative}
Let $K\subset \mathbb{R}^n$ be a convex body, fix $\nu \in \mathbb{S}^{n-1}$, and let $C(\nu,p)$ denote the corresponding slice. For $u\in \nu^\perp$, define
\begin{equation}\label{eq:inf-rep-fp}
f(p):=h_\nu(u,p) =\inf_{\lambda\in\mathbb{R}}
\bigl(h_K(u+\lambda \nu)-\lambda p\bigr),
\end{equation}
whenever $C(\nu,p)\neq \varnothing$. Let $p_0\in \mathbb{R}$, and assume that the infimum in \eqref{eq:inf-rep-fp} is attained by a unique minimizer
$\lambda_0^*=\lambda_0^*(u,\nu,p_0)$ in a neighborhood of $p_0$, and that $f$ is differentiable at $p_0$. Then
\begin{equation}\label{eq:lambda-derivative}
f'(p_0)=
\frac{\partial}{\partial p}h_\nu(u,p_0)
=
-\lambda_0^*.
\end{equation}
\end{proposition}

\begin{proof}
Set
\[
F_{u,p}(\lambda):=h_K(u+\lambda \nu)-\lambda p,
\qquad
f(p)=\inf_{\lambda\in\mathbb{R}}F_{u,p}(\lambda).
\]
Let $\lambda_0^*$ be the unique minimizer at the given value of $p_0$, so that
\[
f(p_0)=F_{u,p_0}(\lambda_0^*).
\]
Note that $\lambda$ depends on $p$, so we cannot differentiate $f$ with respect to $p$ directly to obtain the derivative of $f$ at $p_0$. We therefore first prove the bound for the difference quotient. For any $\Delta p\in\mathbb{R}$,
since $f(p_0+\Delta p)$ is the infimum of $F_{u,p_0+\Delta p}(\lambda)$ over $\lambda$, at $\lambda=\lambda_0^*$ we have
\[
f(p_0+\Delta p)
\le
F_{u, p_0+\Delta p}(\lambda_0^*)
=
h_K(u+\lambda_0^* \nu)-\lambda_0 ^*(p_0+\Delta p).
\]
Subtracting $f(p_0)=h_K(u+\lambda_0^* \nu)-\lambda_0^*p_0$, we obtain, for any $\Delta p\in \mathbb{R}$,
\[
f(p_0+\Delta p)-f(p_0)\le -\lambda_0^*\,\Delta p.
\]
Hence
\begin{equation}\label{eq:upper-dq}
\frac{f(p_0+\Delta p)-f(p_0)}{\Delta p}\le -\lambda_0^*
\qquad\text{if }\Delta p>0,
\end{equation}
while the inequality reverses after division if $\Delta p<0$.

Next, let $\lambda_{\Delta p}^*$ be a minimizer for $f(p_0+\Delta p)$; thus
\[
f(p_0+\Delta p)=F_{u, p_0+\Delta p}(\lambda_{\Delta p}^*)=h_K(u+\lambda^*_{\Delta p}\nu)-\lambda^*_{\Delta p}(p_0+\Delta p).
\]
Evaluating $f(p_0)$ at the same $\lambda_{\Delta p}^*$ gives
\[
f(p_0)=\inf_\lambda F_{u,p_0}(\lambda) \le F_{u,p_0}(\lambda_{\Delta p}^*)
=
h_K(u+\lambda_{\Delta p}^* \nu)-\lambda_{\Delta p}^* p_0.
\]
Subtracting from the identity for $f(p_0+\Delta p)$ yields
\[
f(p_0+\Delta p)-f(p_0)\ge -\lambda_{\Delta p}^* \  \Delta p.
\]
Therefore
\begin{equation}\label{eq:lower-dq}
\frac{f(p_0+\Delta p)-f(p_0)}{\Delta p}\ge -\lambda_{\Delta p}^*
\qquad\text{if }\Delta p>0,
\end{equation}
again with reversed inequality for $\Delta p<0$.

Now use the uniqueness of the minimizer $\lambda_0^*$ at $p_0$. Since $f$ is differentiable at $p_0$, it is in
particular continuous there, and the standard uniqueness argument for minimizers implies that any
sequence $\Delta p_j\to 0$ satisfies
\[
\lambda_{\Delta p_j}^*\to \lambda^*_0.
\]
Passing to the limit in \eqref{eq:upper-dq} and \eqref{eq:lower-dq} from the right and left of the inequalities, we get
$
f'(p_0)=-\lambda_0^*$. This proves \eqref{eq:lambda-derivative}.
\end{proof}

When \(K\) is a strictly convex body with \(C^1\)-boundary, Lemma~\ref{lem:unique-minimizer} ensures the existence and uniqueness of the minimizer \(\lambda\) in the infimal representation \eqref{infhk}, at least for interior slicing heights. Moreover, at every point where \(h_\nu(u,p)\) is differentiable with respect to \(p\), Proposition~\ref{prop:lambda-derivative} gives $\partial_p h_\nu(u,p)=-\lambda$.
Substituting this $\lambda$ into \eqref{infhk} yields a pointwise recovery formula for \(h_K\) in terms of the slicing support function \(h_\nu\), as stated below.

\begin{theorem}[Recovery formula for strictly convex bodies]\label{thm:recovery-strict}
Let $K\subset \mathbb{R}^n$ be a strictly convex body with $C^1$-boundary, let $\nu \in \mathbb{S}^{n-1}$, and $C(\nu,p)$ be the corresponding slice. Let $u\in \nu^\perp$ with $u\neq0$, and assume $m <p<M$, where $m$ and $M$ are defined in \eqref{height}. Then
\begin{equation}\label{eq:recovery-final}
h_K\bigl(u-\partial_p h_\nu(u,p)\,\nu\bigr)
=
h_\nu(u,p)-p\,\partial_p h_\nu(u,p).
\end{equation}
\end{theorem}

\section{Monge-Amp\`ere-Type Structure of Slicing Support Functions}\label{sec4}
When $K\subset\mathbb{R}^3$, the representation \eqref{infhk} suggests a natural Monge-Amp\`ere-type structure.
We now derive this structure explicitly.

Fix a slicing direction \(\nu\in \mathbb{S}^2\), and choose an oriented orthonormal
basis \(\{e_1,e_2\}\) of the plane \(\nu^\perp\). On each slice \(C(\nu,p)\),
the directional variable \(u\in \mathbb{S}^1\subset \nu^\perp\) can be parametrized by
\begin{align}\label{thetavariable}
u(\theta)=\cos\theta\,e_1+\sin\theta\,e_2,
\qquad
u^\perp(\theta)=-\sin\theta\,e_1+\cos\theta\,e_2 .
\end{align}
Thus the slicing support function
\[
H(\theta,p):=h_\nu(u(\theta),p)
\]
may be regarded as a function on \(\mathbb{S}^1\times I_\nu\), where \(I_\nu=(m,M)\). We define the Monge-Amp\`ere-type operator
\begin{align}\label{ahoperator}
\begin{split}
\mathcal{A}[H]
&:=
\det
\begin{pmatrix}
    H_{\theta\theta}+H & H_{\theta p} \\
    H_{\theta p}      & H_{pp}
\end{pmatrix}  \\
&=
(H_{\theta\theta}+H)H_{pp}-(H_{\theta p})^2,
\end{split}
\end{align}
where subscripts denote partial derivatives with respect to the indicated
variables.

The following theorem expresses $\mathcal{A}[H]$ in terms of the classical support function $h_K$ of $K$. Note that the $C^2$ regularity of $\partial K$, together with its positive Gaussian curvature, implies that
$h_K\in C^2(\mathbb R^3\setminus\{0\})$. Hence $F_{\lambda\lambda}(\theta,\lambda)$ exists for every $(\theta,\lambda)$, and the positive Gaussian curvature further implies
$F_{\lambda\lambda}(\theta,\lambda)>0$.

%Note that the assumption of positive Gaussian curvature on $\partial K$ ensures that $F_{\lambda\lambda}(\theta,\lambda)$ is well defined for all $(\theta,\lambda)$.

\begin{theorem}[Cylindrical Monge-Amp\`ere identity]\label{cylinder}
Let $K\subset \mathbb R^3$ be a strictly convex body with $C^2$-boundary and positive Gaussian curvature on $\partial K$. Fix $\nu\in \mathbb{S}^2$, and choose an oriented
orthonormal basis $\{e_1,e_2\}$ of $\nu^\perp$. Let $m, M$ be the heights of $K$ as in \eqref{height} and let $u(\theta)$, $u^\perp(\theta)$ be defined as in \eqref{thetavariable}. Define
\[
F(\theta,\lambda):=h_K\bigl(u(\theta)+\lambda\nu\bigr).
\]
For $p\in(m,M)$, write the slicing support function as
\begin{equation}\label{defH}
H(\theta,p)
=
h_\nu(u(\theta),p)
=
\inf_{\lambda\in\mathbb R}
\bigl(F(\theta,\lambda)-\lambda p\bigr).
\end{equation}
For each $(\theta,p)\in \mathbb{S}^1\times(m,M)$, let
$\lambda=\lambda(\theta,p)$ denote the unique minimizer in \eqref{defH}.
Then the Monge-Amp\`ere-type operator $\mathcal A[H]$ in \eqref{ahoperator} satisfies
\begin{equation}\label{ma-correct}
\mathcal A[H]
=
-\frac{F_{\theta\theta}+F-\lambda F_\lambda}{F_{\lambda\lambda}}
=
-\frac{F_{\theta\theta}+F-\lambda p}{F_{\lambda\lambda}},
\end{equation}
where all derivatives of $F$ are evaluated at
$(\theta,\lambda(\theta,p))$. Moreover,
\[
\mathcal A[H]<0
\qquad
\text{for all }(\theta,p)\in \mathbb{S}^1\times(m,M).
\]
\end{theorem}

\begin{proof}
By Lemma \ref{lem:unique-minimizer} and Proposition~\ref{prop:lambda-derivative} with \eqref{defH}, there exists a unique function $\lambda=\lambda(\theta,p)$ such that
\begin{align}\label{hplambda}
H(\theta,p)=F(\theta,\lambda(\theta,p))-\lambda(\theta,p)\,p;
\end{align}
in addition, the first-order condition for the minimizer of $F(\theta, \lambda)-\lambda p$ gives
\begin{equation}\label{FOC-correct}
F_\lambda(\theta,\lambda(\theta,p))=p.
\end{equation}

We first compute the derivatives of $H$. Differentiating
\eqref{hplambda} with respect to $\theta$, and using the chain rule, gives
\[
H_\theta
=
F_\theta+F_\lambda\lambda_\theta-p\lambda_\theta.
\]
By \eqref{FOC-correct}, the last two terms cancel and
\begin{equation}\label{Htheta}
H_\theta=F_\theta.
\end{equation}

Next, differentiating \eqref{Htheta} with respect to $\theta$ once more, we obtain
\[
H_{\theta\theta}
=
F_{\theta\theta}+F_{\theta\lambda}\lambda_\theta.
\]
To determine $\lambda_\theta$, we differentiate \eqref{FOC-correct} with respect to $\theta$ and use $F_{\lambda\theta}=F_{\theta\lambda}$ to obtain
$$
F_{\lambda\theta}+F_{\lambda\lambda}\lambda_\theta=0,
$$
which implies $
\lambda_\theta=-\frac{F_{\theta\lambda}}{F_{\lambda\lambda}}
$. Note that $F_{\lambda\lambda}>0$ by the strict convexity assumption of $K$.
Substituting this into the formula for $H_{\theta\theta}$ gives
\begin{equation}\label{Hthetatheta}
H_{\theta\theta}
=
F_{\theta\theta}
-\frac{F_{\theta\lambda}^2}{F_{\lambda\lambda}}.
\end{equation}

Next, differentiating \eqref{hplambda} with respect to $p$, we obtain
\[
H_p
=
F_\lambda\lambda_p-\lambda-p\lambda_p.
\]
Again using \eqref{FOC-correct}, the first and third terms cancel, so
\begin{equation}\label{Hp}
H_p=-\lambda.
\end{equation}

Differentiating \eqref{FOC-correct} with respect to $p$ gives $
F_{\lambda\lambda}\lambda_p=1$,
hence $\lambda_p=\frac{1}{F_{\lambda\lambda}}$.
Therefore, differentiating \eqref{Hp} with respect to $p$, we get
\begin{equation}\label{Hpp}
H_{pp}=-\lambda_p=-\frac{1}{F_{\lambda\lambda}}.
\end{equation}

Differentiating \eqref{Hp} with respect to $\theta$, we obtain
\begin{align}\label{Hthetap}
H_{\theta p}=-\lambda_\theta
=
\frac{F_{\theta\lambda}}{F_{\lambda\lambda}}.
\end{align}

From \eqref{Hthetatheta}, \eqref{hplambda}, and \eqref{FOC-correct}, we obtain
\begin{equation}\label{Hthetatheta-plus-H}
H_{\theta\theta}+H
=
F_{\theta\theta}
-\frac{F_{\theta\lambda}^2}{F_{\lambda\lambda}}
+
F-\lambda F_\lambda.
\end{equation}

Substituting \eqref{Hpp}, \eqref{Hthetap}, and \eqref{Hthetatheta-plus-H} into \eqref{ahoperator}, and using \eqref{FOC-correct}, we find
\begin{align*}
\mathcal A[H]
&=
(H_{\theta\theta}+H)H_{pp}-H_{\theta p}^2 \\
&=
\left(
F_{\theta\theta}
-\frac{F_{\theta\lambda}^2}{F_{\lambda\lambda}}
+
F-\lambda F_\lambda
\right)
\left(-\frac{1}{F_{\lambda\lambda}}\right)
-
\left(\frac{F_{\theta\lambda}}{F_{\lambda\lambda}}\right)^2\\
&
=-\frac{F_{\theta\theta}+F-\lambda p}{F_{\lambda\lambda}}.
\end{align*}
This proves \eqref{ma-correct}.

It remains to prove the sign of $H_{\theta\theta}+H$. For each fixed $p\in (m, M)$, since the slice $C(\nu,p)$ is a strictly convex planar body with $C^2$-boundary, its support function $H(\theta,p)$ satisfies
\[
H_{\theta\theta}+H>0.
\]
Also, by \eqref{Hpp} and the strict convexity in the $\lambda$-variable,
$
F_{\lambda\lambda}>0
$ implies that
$H_{pp}<0.
$
Therefore,
\[
\mathcal A[H]
=
(H_{\theta\theta}+H)H_{pp}-H_{\theta p}^2
<0.
\]
This completes the proof.
\end{proof}

We observe that the right-hand side of equation \eqref{ma-correct} has a natural geometric interpretation. For this purpose, we first recall the following notation on the unit sphere \(\mathbb{S}^2 \subset \mathbb{R}^3\). Let $\nabla^{\mathbb{S}^2}$ be the Levi-Civita connection on
$\mathbb{S}^2$ induced by the Euclidean metric in $\mathbb{R}^3$. For tangent vector fields $X,Y$ on $\mathbb{S}^2$, we have that
\[
\nabla^{\mathbb{S}^2}_X Y=(D_XY)^\top,
\]
where $D$ is the Euclidean connection and $(\cdot)^\top$ denotes projection
onto $T_\omega \mathbb{S}^2=\{v\in\mathbb{R}^3:v\cdot\omega=0\}$ for any $\omega\in \mathbb{S}^2$. For a smooth function $h:\mathbb{S}^2\to\mathbb{R}$, define its spherical Hessian by
\[
\nabla_{\mathbb{S}^2}^2 h(X,Y)
=
X(Yh)-(\nabla^{\mathbb{S}^2}_X Y)h,
\qquad X,Y\in T_\omega\mathbb{S}^2.
\]
The spherical curvature matrix of $K$ at $\omega$ is then
\begin{equation}\label{curvmatrix}
Q_\omega:=Q_{\bar h_K}(\omega)
=
\nabla_{\mathbb{S}^2}^2\bar h_K(\omega)
+\bar h_K(\omega)I,
\end{equation}
where $I$ denotes the identity operator on $T_\omega \mathbb{S}^2$ and \(\bar h_K=h_K|_{\mathbb{S}^2}\).

In the next corollary, we shall use the notation $A[u,v]$ to denote the evaluation of a bilinear
form $A$ on the pair $(u,v)$. Equivalently, after choosing a basis of
$T_\omega \mathbb{S}^2$, if $[A]$ is the matrix representation of $A$ and $[u]$,
$[v]$ are the corresponding column vectors, then
$
A[u,v]=[u]^T[A][v]$.

\begin{corollary}\label{prop:second-derivation-rhs}
Assume the same notation and hypotheses on $K$ and $\partial K$ as in Theorem \ref{cylinder}. Let
\[
v:=u(\theta)+\lambda \nu,
\qquad
\omega:=\frac{v}{|v|}\in \mathbb{S}^2,
\]
where \(\lambda=\lambda(\theta,p)\) denotes the minimizer attained in \(\eqref{defH}\). Define the spherical curvature matrix $Q_\omega$ by \eqref{curvmatrix}.
Then
\begin{equation}\label{eq:rhs-directional-radii}
\mathcal A[H]
=
-(1+\lambda^2)\,
\frac{Q_\omega[u^\perp,u^\perp]}{Q_\omega[\tau,\tau]},
\end{equation}
where
\[
\tau
:=
\frac{\nu-\lambda u(\theta)}{\sqrt{1+\lambda^2}}
\in T_\omega \mathbb{S}^2
\]
is the unit tangent vector orthogonal to \(u^\perp\).

In particular, if \(u^\perp\) and \(\tau\) are principal directions at \(\omega\), then
\[
Q_\omega[u^\perp,u^\perp]=r_1,
\qquad
Q_\omega[\tau,\tau]=r_2,
\]
where \(r_i=1/\kappa_i\) are the corresponding principal radii of curvature. Hence
\begin{equation}\label{eq:principal-curvature-ratio}
\mathcal A[H]
=
-(1+\lambda^2)\frac{r_1}{r_2}
=
-(1+\lambda^2)\frac{\kappa_2}{\kappa_1}.
\end{equation}
\end{corollary}

\begin{proof}
First, we differentiate $F$ along the curve $u(\theta)$ while holding $\lambda$ fixed, and so
\[
F_\theta
=
\nabla h_K(v)\cdot u^\perp.
\]
Differentiating again with respect to $\theta$ and using $(u^\perp)_\theta=-u$,
we obtain
\[
F_{\theta\theta}
=
D^2 h_K(v)[u^\perp,u^\perp]-\nabla h_K(v)\cdot u.
\]
Now Euler's identity for the $1$-homogeneous support function $h_K$ gives
\[
\nabla h_K(v)\cdot v=h_K(v)=F.
\]
Since $v=u+\lambda \nu$, we have
\[
\nabla h_K(v)\cdot u
=
\nabla h_K(v)\cdot (v-\lambda \nu)
=
F-\lambda\,\nabla h_K(v)\cdot \nu.
\]
But
\begin{align}\label{Flambda}
    F_\lambda=\nabla h_K(v)\cdot \nu,
\end{align} and so
$\nabla h_K(v)\cdot u = F-\lambda F_\lambda$. Substituting this into the expression for $F_{\theta\theta}$, we obtain
\begin{align}\label{eq:DuF-contracted}
D^2 h_K(v)[u^\perp,u^\perp] =F_{\theta\theta}+F-\lambda F_\lambda,
\end{align}
which is the numerator of the right-hand side of \eqref{ma-correct}.

It is well known that for every tangent vector $X \in T_\omega \mathbb{S}^2$ (a proof can be found in Lemma~\ref{lem:euclid-sphere} in the Appendix), for any $\rho>0$,
\begin{equation}\label{eq:euclid-sphere-relation}
D^2 h_K(\rho\omega)[X,X]
=
\frac1{\rho}\,Q_\omega[X,X].
\end{equation}
Since $u^\perp\perp u$ and $u^\perp\perp \nu$, $u^\perp\cdot v=0$, and hence $u^\perp\in T_\omega \mathbb{S}^2$. Set $\rho:=|v|=\sqrt{1+\lambda^2}$.
Then, by \eqref{eq:euclid-sphere-relation}, we have
\begin{equation}\label{eq:numerator-spherical}
D^2 h_K(v)[u^\perp,u^\perp]
=
\frac1{\rho}\,Q_\omega[u^\perp,u^\perp].
\end{equation}

Next, we compute the denominator. Note that $\nu$ is not tangent to $\mathbb{S}^2$ at $\omega$. Decompose it into tangential and radial parts,
\[
\nu =\nu_T+(\nu\cdot\omega)\omega.
\]
Since $D^2 h_K(v)[\omega,\cdot]=0$ for a $1$-homogeneous function, only the tangential part contributes,
\[
D^2 h_K(v)[\nu,\nu]=D^2 h_K(v)[\nu_T,\nu_T].
\]
A direct computation shows that
\[
\nu \cdot\omega=\frac{\lambda}{\rho},
\qquad
\nu_T=\nu-\frac{\lambda}{\rho}\omega
=
\frac{\nu-\lambda u}{1+\lambda^2}.
\]
Define
\[
\tau:=\rho\, \nu_T=\frac{\nu-\lambda u}{\rho}.
\]
Then $\tau\in T_\omega \mathbb{S}^2$ and $|\tau|=1$. Hence $\nu_T=\frac1{\rho}\tau$, and so
\[
D^2 h_K(v)[\nu,\nu]
=
D^2 h_K(v)[\nu_T,\nu_T]
=
\frac1{\rho^2}D^2 h_K(v)[\tau,\tau].
\]
Applying \eqref{eq:euclid-sphere-relation} again,
\begin{equation}\label{eq:denominator-spherical}
D^2 h_K(v)[\nu,\nu]
=
\frac1{\rho^3}Q_\omega[\tau,\tau].
\end{equation}
Finally, by \eqref{Flambda} we have that
\begin{align}\label{Fll}
F_{\lambda\lambda}
=
D^2 h_K(v)[\nu,\nu].
\end{align}

Using \eqref{eq:DuF-contracted}, \eqref{eq:numerator-spherical}, \eqref{eq:denominator-spherical}, and \eqref{Fll}, we conclude that
\[
\mathcal A[H]
=
-
\frac{D^2 h_K(v)[u^\perp,u^\perp]}{D^2 h_K(v)[\nu,\nu]}
=
-\rho^2\frac{Q_\omega[u^\perp,u^\perp]}{Q_\omega[\tau,\tau]}.
\]
Since $\rho^2=1+\lambda^2$, this yields \eqref{eq:rhs-directional-radii}.

Since $\partial K$ has positive Gaussian curvature, the Gauss map $N$ is a local diffeomorphism and the Weingarten map $W=-dN$ is nonsingular. With respect to the outward unit normal, we adopt the convention that the positive principal curvatures $\kappa_1,\kappa_2>0$ are the eigenvalues of $dN$; equivalently, the eigenvalues of $W$ are $-\kappa_1$ and $-\kappa_2$. Since
\[
Q_\omega=(dN)^{-1}=-W^{-1},
\]
the eigenvalues of $Q_\omega$ are the principal radii of curvature
\[
r_1=\frac{1}{\kappa_1},
\qquad
r_2=\frac{1}{\kappa_2}.
\]
In particular, when $u^\perp$ and $\tau$ are principal directions, one has
\[
Q_\omega[u^\perp,u^\perp]=r_1,
\qquad
Q_\omega[\tau,\tau]=r_2,
\]
and hence
\[
\mathcal A[H]
=
-(1+\lambda^2)\frac{r_1}{r_2}
=
-(1+\lambda^2)\frac{\kappa_2}{\kappa_1}.
\]
This proves the corollary.
\end{proof}

In Corollary~\ref{prop:second-derivation-rhs}, it concerns the case when $\tau$ and $u^\perp$ are principal directions.
The final result of this section characterizes when this condition holds everywhere except for the poles, if and only if, up to translation, $K$ is a body of revolution about an axis parallel to $\nu$.

\begin{proposition}\label{prop:principal-revolution}
Let $K \subset \mathbb{R}^3$ be a convex body with $C^2$-boundary and positive Gaussian curvature, and let $h_K \in C^2(\mathbb{S}^2)$ be its support function. Fix $\nu \in \mathbb{S}^2$. For each $\omega \in \mathbb{S}^2 \setminus \{\pm\nu\}$, write
\[
\omega = \sin\alpha \, u(\theta) + \cos\alpha \, \nu,
\qquad
0 < \alpha < \pi,
\]
where $u(\theta) \in \mathbb{S}^2 \cap \nu^\perp$, and define an orthonormal basis $\{u^\perp, \tau\}$ of $T_\omega \mathbb{S}^2$, where
\[
u^\perp(\theta) := \frac{\partial u}{\partial\theta},
\qquad
\tau := \sin\alpha \, \nu - \cos\alpha \, u(\theta).
\]
Then the following statements are equivalent:
\begin{enumerate}
\item For every $\omega \in \mathbb{S}^2 \setminus \{\pm\nu\}$, the vectors $u^\perp$ and $\tau$ are principal directions of $\partial K$ at $N^{-1}(\omega)$, where $N$ is the Gauss map of $K$.
\item There exist a function $\varphi: [-1,1] \to \mathbb{R}$ and a vector $a \in \mathbb{R}^3$ such that
\[
h_K(\omega) = \varphi(\omega \cdot \nu) + a \cdot \omega,
\qquad
\omega \in \mathbb{S}^2.
\]
\item Up to a translation, $K$ is a body of revolution about an axis parallel to $\nu$.
\end{enumerate}
\end{proposition}

\begin{proof}
We first prove that $(1)\Rightarrow(2)$. Since $\partial K$ has positive Gaussian curvature, $Q_\omega$ is positive definite and its eigenvectors correspond to the principal directions of $\partial K$ at $N^{-1}(\omega)$.
Because $\{u^\perp, \tau\}$ is an orthonormal basis of $T_\omega \mathbb{S}^2$ and $Q_\omega$ is symmetric, the vectors $u^\perp$ and $\tau$ are principal directions if and only if
$Q_\omega [ u^\perp, \tau] = 0$.
In addition, the orthonormality of $u^\perp$ and $\tau$ implies that
$Q_\omega [u^\perp, \tau] = \nabla_{\mathbb{S}^2}^2h_K(u^\perp, \tau)$.
Set
\[
E_\alpha := \frac{\partial\omega}{\partial\alpha} = \cos\alpha \, u(\theta) - \sin\alpha \, \nu = -\tau
\]
and
\[
E_\theta := \frac{1}{\sin\alpha} \frac{\partial\omega}{\partial\theta} = u^\perp(\theta).
\]
The standard round metric on $\mathbb{S}^2$ is $d\alpha^2 + \sin^2\alpha \ d\theta^2$, and its Levi-Civita connection satisfies
$D_{E_\theta}E_\alpha =\nabla^{\mathbb{S}^2}_{E_\theta}E_\alpha= \cot\alpha \, E_\theta$.
Hence
\[
\begin{aligned}
\nabla_{\mathbb{S}^2}^2h_K(E_\theta, E_\alpha)
&= E_\theta(E_\alpha h_K) - (\nabla^{\mathbb{S}^2}_{E_\theta}E_\alpha)h_K \\
&= \frac{1}{\sin\alpha}
\left(
\frac{\partial^2h_K}{\partial\theta\partial\alpha} - \cot\alpha \frac{\partial h_K}{\partial\theta}
\right).
\end{aligned}
\]
Consequently, $Q_\omega [ u^\perp, \tau] = 0$ is equivalent to
\begin{equation}\label{eq:mixed-spherical-hessian}
\frac{\partial^2h_K}{\partial\theta\partial\alpha} - \cot\alpha \frac{\partial h_K}{\partial\theta} = 0.
\end{equation}

Equation~\eqref{eq:mixed-spherical-hessian} can be rewritten as
\[
\frac{\partial}{\partial\alpha}
\left(
\frac{1}{\sin\alpha}
\frac{\partial h_K}{\partial\theta}
\right) = 0.
\]
Therefore, there exists a $2\pi$-periodic function $b(\theta)$ such that
\[
\frac{\partial h_K}{\partial\theta} = \sin\alpha \, b(\theta).
\]
Integrating with respect to $\theta$, we obtain
\begin{equation}\label{eq:local-support-form}
h_K(\theta,\alpha) = A(\alpha) + \sin\alpha \, B(\theta)
\end{equation}
for suitable functions $A$ and $B$.
Define the rotational average
\[
\overline{h}(\alpha) = \frac{1}{2\pi} \int_0^{2\pi} h_K(\theta,\alpha) \, d\theta.
\]
Then $\overline{h}$ is rotationally symmetric about the $\nu$-axis, and
\[
h_K(\theta,\alpha) - \overline{h}(\alpha) = \sin\alpha \, \widetilde{B}(\theta),
\]
where $\widetilde{B}$ has zero average.
Near the north pole, introduce the local coordinates
\[
x = \sin\alpha\cos\theta,
\qquad
y = \sin\alpha\sin\theta.
\]
The function
\[
(x,y) \longmapsto \sqrt{x^2+y^2} \, \widetilde{B}\bigl(\arg(x+iy)\bigr)
\]
is differentiable at $(0,0)$ and positively homogeneous of degree one.
Every differentiable, positively homogeneous function of degree one is linear. Hence there exist constants $a_1, a_2 \in \mathbb{R}$ such that
\[
\sin\alpha \, \widetilde{B}(\theta) = a_1\sin\alpha\cos\theta + a_2\sin\alpha\sin\theta.
\]
Choose an orthonormal basis $\{v_1, v_2, \nu\}$ such that
\[
u(\theta) = \cos\theta \, v_1 + \sin\theta \, v_2,
\]
and set $a_\perp := a_1v_1 + a_2v_2$.
Since $a_\perp \cdot \omega = a_1\sin\alpha\cos\theta + a_2\sin\alpha\sin\theta$, we obtain
\[
h_K(\omega) = \varphi(\omega\cdot\nu) + a_\perp \cdot \omega
\]
for some function $\varphi$.

Next, we show that $(2)\Rightarrow(3)$. Assume that $(2)$ holds. The support function of the translated body $K-a$ is
\[
h_{K-a}(\omega) = h_K(\omega) - a \cdot \omega = \varphi(\omega\cdot\nu).
\]
Thus, for every rotation $R \in SO(3)$ satisfying $R\nu = \nu$,
\[
h_{K-a}(R\omega) = h_{K-a}(\omega).
\]
Since a convex body is uniquely determined by its support function, it follows that
$R(K-a) = K-a$.
Hence $K-a$ is a body of revolution about the $\nu$-axis.

Finally, we show that $(3)\Rightarrow(1)$. Suppose that, up to a translation, $K$ is a body of revolution about an axis parallel to $\nu$. Then its support function has the form $h_K(\omega) = \varphi(\omega\cdot\nu) + a \cdot \omega$. In the coordinates $(\theta,\alpha)$, a direct calculation gives
\[
\frac{\partial^2h_K}{\partial\theta\partial\alpha} - \cot\alpha \frac{\partial h_K}{\partial\theta} = 0.
\]
It follows that $Q_\omega [u^\perp, \tau] = 0$.
Since $Q_\omega$ is symmetric, the orthonormal vectors $u^\perp$ and $\tau$ are eigenvectors of $Q_\omega$. Therefore, they correspond to the principal directions of $\partial K$ at $N^{-1}(\omega)$.
\end{proof}

\section{Iterated Slicing Support Functions in the Higher-Codimensional Setting}\label{sec5}
The geometry of a convex body \(K \subset \mathbb{R}^n\) can be decomposed through successive slicing. In this section, we develop a recursive framework for the associated slicing support functions, derive the corresponding determinant identities, and show that they exhibit a nested Monge-Amp\`ere structure.

Let \(K \subset \mathbb{R}^n\) be a convex body and $\{\nu_1,\dots,\nu_k\}$ be an orthonormal family in \(\mathbb{R}^n\), representing \(k\) successive slicing directions, where \(1\leq k\leq n-1\). Define the \(k\)-dimensional slicing subspace by $V_k=\operatorname{span}\{\nu_1,\dots,\nu_k\}$, and let $E_{n-k}=V_k^\perp$ be its orthogonal complement in $\mathbb{R}^n$. Thus \(E_{n-k}\) is an \((n-k)\)-dimensional subspace on which the $k$-th slicing support function, introduced in \eqref{eq:global_inf}, is defined. Consequently, every vector \(v\in\mathbb{R}^n\) admits the unique orthogonal decomposition
\[
v=u+\sum_{j=1}^k \lambda_j\nu_j,
\qquad
u\in E_{n-k},\quad \lambda_j\in\mathbb{R}.
\]

\begin{definition}[Iterated Slicing Support Function]\label{defiterate}
Let $\boldsymbol{\nu}_k = (\nu_1, \dots, \nu_k)$ be an ordered $k$-tuple of slicing directions, and let $\mathbf{p}_k = (p_1, \dots, p_k) \in \mathbb{R}^k$. When the context is clear, we write $\boldsymbol{\nu}$ instead of $\boldsymbol{\nu}_k$.

For a convex body $K \subset \mathbb{R}^n$, we recall that the $k$-fold sliced section $K^{(k)}(\boldsymbol{\nu}_k, \mathbf{p}_k)$ is defined as in \eqref{ktimesslide}.
Note that when $k=1$, $K^{(1)}(\nu_1, p_1)=C(\nu_1, p_1)$, the section of $K$ sliced by the affine hyperplane parallel to $\nu_1^\perp$. Whenever $K^{(k)}(\boldsymbol{\nu}_k, \mathbf{p}_k) \neq \emptyset$, the $k$-th \emph{iterated slicing support function} $h_{\boldsymbol{\nu}_k}^{(k)}: E_{n-k} \times \mathbb{R}^k \to \mathbb{R}$ is defined as the support function of $K^{(k)}(\boldsymbol{\nu}_k, \mathbf{p}_k)$ restricted to $E_{n-k}$, given by
\[
h_{\boldsymbol{\nu}_k}^{(k)}(u, \mathbf{p}_k) = \sup_{x \in K^{(k)}(\boldsymbol{\nu}_k, \mathbf{p}_k)} x \cdot u, \quad \text{for } u \in E_{n-k}.
\]
Using the same convention as for $k=1$, we set
$h_{\boldsymbol{\nu}_k}^{(k)}(u,\mathbf p_k)=-\infty$ whenever $K^{(k)}(\boldsymbol{\nu}_k,\mathbf p_k)=\emptyset.
$
\end{definition}
We now derive the following recursive formula for the slicing support functions for $1\leq k\leq n-1$.

\begin{theorem}[Iterated Representation]\label{iteratedrepre}
Let $K$ be a convex body in $\mathbb R^n$. For every
$1\le k\le n-1$, the $k$-th slicing support function satisfies
\eqref{eq:global_inf}. Moreover, for $2\le k\le n-1$, it satisfies
\begin{equation}\label{eq:recursive_inf}
h_{(\boldsymbol{\tilde\nu},\nu_k)}^{(k)}(u,\mathbf p_k)
=\inf_{\lambda_k\in\mathbb R}
\left(
h_{\boldsymbol{\tilde\nu}}^{(k-1)}
(u+\lambda_k\nu_k,\mathbf p_{k-1})-\lambda_kp_k
\right),
\end{equation}
where $\boldsymbol{\tilde\nu}=(\nu_1,\ldots,\nu_{k-1})$ and $(\boldsymbol{\tilde{\nu}},\nu_k)=\boldsymbol{\nu}_k$.
\end{theorem}

\begin{proof}
For \(k=1\), the assertion is exactly Proposition~\ref{basic}. Namely, for \(u\in E_{n-1}=\nu_1^\perp\),
\[
h_{\nu_1}^{(1)}(u,p_1) = \inf_{\lambda_1\in\mathbb R} \left( h_K(u+\lambda_1\nu_1)-\lambda_1p_1 \right).
\]

Now assume that the result holds for \(k-1\). Let $\boldsymbol{\tilde\nu}=(\nu_1,\dots,\nu_{k-1})$ and $\mathbf p_{k-1}=(p_1,\dots,p_{k-1})$.
Then the \((k-1)\)-fold slice
\[
K^{(k-1)}(\boldsymbol{\tilde\nu}, \mathbf p_{k-1}) = K\cap \{x \in \mathbb{R}^n \mid x\cdot \nu_j=p_j,\ j=1,\dots,k-1\}
\]
lies in the $(n-k+1)$-dimensional affine subspace
\[
A_{k-1}(\boldsymbol{\tilde\nu}, \mathbf p_{k-1}) = \{x\in\mathbb R^n \mid x\cdot \nu_j=p_j,\ j=1,\dots,k-1\},
\]
whose direction space is \(E_{n-k+1}\).

Since \(\nu_k\in E_{n-k+1}\), we may apply the one-step slicing formula inside this affine space. Thus, for \(u\in E_{n-k}\), we obtain
\[
h_{(\boldsymbol{\tilde\nu},\nu_k)}^{(k)}(u,\mathbf p_k) = \inf_{\lambda_k\in\mathbb R} \left( h_{\boldsymbol{\tilde\nu}}^{(k-1)} (u+\lambda_k\nu_k,\mathbf p_{k-1}) - \lambda_k p_k \right).
\]
This proves the recursive formula \eqref{eq:recursive_inf}.

By the induction hypothesis,
\[
h_{\boldsymbol{\tilde\nu}}^{(k-1)} (u+\lambda_k\nu_k,\mathbf p_{k-1}) = \inf_{\boldsymbol{\tilde\lambda}\in\mathbb R^{k-1}} \left( h_K\left( u+\lambda_k\nu_k+\sum_{j=1}^{k-1}\lambda_j\nu_j \right) - \sum_{j=1}^{k-1}\lambda_jp_j \right),
\]
where \(\boldsymbol{\tilde\lambda}=(\lambda_1,\dots,\lambda_{k-1})\). Substituting this into the recursive formula \eqref{eq:recursive_inf} gives
\[
h_{(\boldsymbol{\tilde\nu},\nu_k)}^{(k)}(u,\mathbf p_k) = \inf_{\lambda_k\in\mathbb R} \inf_{\boldsymbol{\tilde\lambda}\in\mathbb R^{k-1}} \left( h_K\left( u+\sum_{j=1}^{k}\lambda_j\nu_j \right) - \sum_{j=1}^{k}\lambda_jp_j \right).
\]
Since the variables \((\boldsymbol{\tilde\lambda},\lambda_k)\) range over \(\mathbb R^{k-1}\times\mathbb R\cong\mathbb R^k\), we can combine the two infima into a single infimum and obtain
\[
h_{\boldsymbol{\nu}}^{(k)}(u,\mathbf p_k) = \inf_{\boldsymbol{\lambda}\in\mathbb R^k} \left( h_K\left( u+\sum_{j=1}^{k}\lambda_j\nu_j \right) - \sum_{j=1}^{k}\lambda_jp_j \right).
\]
This establishes the global formula \eqref{eq:global_inf} and completes the induction.
\end{proof}

Finally, we shall prove the nested structure of the Hessian determinant for the slicing support functions.
By the full Hessian of \(h_{\boldsymbol{\nu}}^{(k)}\), we mean the Hessian
with respect to all variables \((u,\mathbf p_k)\in E_{n-k}\times\mathbb R^k\).
More precisely, after choosing an orthonormal basis of \(E_{n-k}\) and writing \(u=(u_1,\dots,u_{n-k})\) and $\mathbf p_k=(p_1,\cdots, p_k)\in \mathbb{R}^k$, the full Hessian of \(h_{\boldsymbol{\nu}}^{(k)}\) is an \(n\times n\) matrix given by
\[
D^2_{(u,\mathbf p_k)}h_{\boldsymbol{\nu}}^{(k)}
=
\begin{pmatrix}
h_{uu} & h_{u\mathbf p}\\
h_{\mathbf p u} & h_{\mathbf p\mathbf p}
\end{pmatrix},
\]
where
\begin{align*} h_{uu}
&= \left( \frac{\partial^2 h_{\boldsymbol{\nu}}^{(k)}} {\partial u_\alpha\,\partial u_\beta} \right)_{1\le \alpha,\beta\le n-k}, \quad h_{u\mathbf p} = \left( \frac{\partial^2 h_{\boldsymbol{\nu}}^{(k)}} {\partial u_\alpha\,\partial p_j} \right)_{\substack{1\le \alpha\le n-k\\ 1\le j\le k}}, \\
h_{\mathbf p u}
&= \left( \frac{\partial^2 h_{\boldsymbol{\nu}}^{(k)}} {\partial p_i\,\partial u_\beta} \right)_{\substack{1\le i\le k\\ 1\le \beta\le n-k}}, \quad h_{\mathbf p\mathbf p} = \left( \frac{\partial^2 h_{\boldsymbol{\nu}}^{(k)}} {\partial p_i\,\partial p_j} \right)_{1\le i,j\le k}.
\end{align*}

\begin{theorem}[Full Hessian Determinant Structure]\label{fullhessian}
Let \(K\subset\mathbb R^n\) be a convex body and define
\[
F(u,\boldsymbol{\lambda})
=
h_K\left(u+\sum_{j=1}^k\lambda_j\nu_j\right),
\qquad
(u,\boldsymbol{\lambda})\in E_{n-k}\times\mathbb R^k.
\]
Assume that \(F\) is \(C^2\) near the point under consideration, that the minimizer
\(\boldsymbol{\lambda}=\boldsymbol{\lambda}(u,\mathbf p_k)\) of
\[
F(u,\boldsymbol{\lambda})-\boldsymbol{\lambda}\cdot\mathbf p_k
\]
is unique and depends smoothly on \((u,\mathbf p_k)\), and that
\(F_{\boldsymbol{\lambda}\boldsymbol{\lambda}}\) is invertible at the minimizer. Then
\begin{equation}\label{eq:full_hessian_det_identity}
\det\left(D^2_{(u,\mathbf p_k)}
h_{\boldsymbol{\nu}}^{(k)}\right)
=
(-1)^k
\frac{\det(F_{uu})}
{\det(F_{\boldsymbol{\lambda}\boldsymbol{\lambda}})}.
\end{equation}
Here all derivatives of \(F\) are evaluated at
\((u,\boldsymbol{\lambda}(u,\mathbf p_k))\).
\end{theorem}

\begin{proof}
Write
\[
A=F_{\boldsymbol{\lambda}\boldsymbol{\lambda}}
=
\left[
\frac{\partial^2 F}{\partial \lambda_i \partial \lambda_j}
\right],
\qquad
B=F_{u\boldsymbol{\lambda}}
=
\left[
\frac{\partial^2 F}{\partial u_\alpha \partial \lambda_i}
\right],
\]
and
\[
C=F_{uu}
=
\left[
\frac{\partial^2 F}{\partial u_\alpha \partial u_\beta}
\right]
\]
for $1\leq i,j\leq k$ and $1 \leq \alpha,\beta \leq n-k$. Throughout, the notation \([M_{ij}]\) denotes the matrix
whose \((i,j)\)-entry is \(M_{ij}\). By symmetry of the Hessian, we note that $F_{\boldsymbol{\lambda}u} = B^T$.

The first-order condition for the minimizer implies $F_{\boldsymbol{\lambda}}(u,\boldsymbol{\lambda})
=
\mathbf p_k.
$
Differentiating this identity with respect to \(\mathbf p_k\) gives
$A\boldsymbol{\lambda}_{\mathbf p}=I$,
and hence $\boldsymbol{\lambda}_{\mathbf p}=A^{-1}$.
Differentiating the same identity with respect to \(u\) gives $F_{\boldsymbol{\lambda}u}
+
A\boldsymbol{\lambda}_{u}=0$,
so that
$\boldsymbol{\lambda}_{u}
=
-A^{-1}B^T$.

Now consider the function $
h_{\boldsymbol{\nu}}^{(k)}(u,\mathbf p_k)
=
F(u,\boldsymbol{\lambda}(u,\mathbf p_k))
-
\boldsymbol{\lambda}(u,\mathbf p_k)\cdot \mathbf p_k$ at the minimizer. Using the chain rule along with the first-order condition
\(F_{\boldsymbol{\lambda}}=\mathbf p_k\), the terms involving derivatives of $\boldsymbol{\lambda}$ cancel out, yielding
\[
h_{\mathbf p}
=
-\boldsymbol{\lambda},
\qquad
h_u
=
F_u.
\]
Differentiating once more yields the second-order blocks:
\[
h_{\mathbf p\mathbf p}
=
-\boldsymbol{\lambda}_{\mathbf p}
=
-A^{-1},
\]
\[
h_{u\mathbf p}
=
F_{u\boldsymbol{\lambda}}\boldsymbol{\lambda}_{\mathbf p}
=
BA^{-1},
\]
and
\[
h_{uu}
=
F_{uu}
+
F_{u\boldsymbol{\lambda}}\boldsymbol{\lambda}_u
=
C-BA^{-1}B^T.
\]
Therefore, the full Hessian matrix of
\(h_{\boldsymbol{\nu}}^{(k)}\) with respect to \((u,\mathbf p_k)\) can be written as
\[
D^2_{(u,\mathbf p_k)}h_{\boldsymbol{\nu}}^{(k)}
=
\begin{pmatrix}
C-BA^{-1}B^T & BA^{-1} \\
A^{-1}B^T & -A^{-1}
\end{pmatrix}.
\]

Recall that the Schur complement formula states that if
\[
\mathcal M=
\begin{pmatrix}
\mathcal P & \mathcal Q\\
\mathcal R & \mathcal S
\end{pmatrix}
\]
and \(\mathcal S\) is invertible, then
$
\det \mathcal M
=
\det(\mathcal S)
\det\left(\mathcal P-\mathcal Q\mathcal S^{-1}\mathcal R\right).
$
We compute the determinant by using the Schur complement formula with respect to the lower-right block \(-A^{-1}\). Since $A$ is a $k\times k$ matrix, we have
\begin{align*}
&\hspace{.5cm}\det\left(D^2_{(u,\mathbf p_k)}h_{\boldsymbol{\nu}}^{(k)}\right) \\
&=
\det(-A^{-1})
\det\left[
\left(C-BA^{-1}B^T\right)
-
\left(BA^{-1}\right)(-A)\left(A^{-1}B^T\right)
\right] \\
&=
\det(-A^{-1})\det(C) =(-1)^k\frac{1}{\det A}\det(C)\\
&=(-1)^k \frac{ \det(F_{uu})}{\det(F_{\boldsymbol{\lambda}\boldsymbol{\lambda}})}.
\end{align*}
This completes the proof.
\end{proof}

\section{Appendix}\label{sec6}
\subsection{Derivation of the Hessian formula for the support function}
Let $K\subset \R^3$ be a convex body whose support function $h_K$ is of class $C^2$ on $\R^3\setminus\{0\}$. Let
$\bar h := h_K|_{\mathbb{S}^2}$.
For each $\omega\in \mathbb{S}^2$, define the symmetric bilinear form
\[
Q_\omega := \nabla^2_{\mathbb{S}^2}\bar h(\omega)+\bar h(\omega)I
\]
on $T_\omega \mathbb{S}^2$, where $I$ denotes the identity map, and $\nabla^2_{\mathbb{S}^2}$ is the spherical Hessian.

We now prove \eqref{eq:euclid-sphere-relation} in Corollary \ref{prop:second-derivation-rhs}.

\begin{lemma}\label{lem:euclid-sphere}
Let $F:\mathbb R^3\setminus\{0\}\to\mathbb R$ be a $C^2$ function
homogeneous of degree $1$. Set
\[
\bar F:=F|_{\mathbb S^2},
\qquad
Q^F_\omega
:=\nabla^2_{\mathbb S^2}\bar F(\omega)+\bar F(\omega)I.
\]
Then, for every $\rho>0$, $\omega\in\mathbb S^2$, and
$X\in T_\omega\mathbb S^2$,
\begin{align}\label{eq:main}
D^2F(\rho\omega)[X,X]
=\frac1\rho Q^F_\omega[X,X].
\end{align}
In particular, for $F=h_K$ one has $Q^F_\omega=Q_\omega$, so this gives
\eqref{eq:euclid-sphere-relation}.
\end{lemma}

\begin{proof}
Fix $\rho>0$, $\omega\in \mathbb{S}^2$, and $X\in T_\omega \mathbb{S}^2$. Let $\gamma(t)\subset \mathbb{S}^2$ be the geodesic satisfying $\gamma(0)=\omega$ and $\gamma'(0)=X$.
Since $|\gamma(t)|^2=1$, by differentiating twice, we have
\[
\gamma(t)\cdot \gamma''(t)=-|\gamma'(t)|^2.
\]
On the other hand, since \(\gamma\) is a geodesic on \(\mathbb{S}^2\), its Euclidean acceleration is normal to \(\mathbb{S}^2\). Hence $\gamma''(t)=a(t)\gamma(t)$
for some scalar function \(a(t)\). Therefore
\[
a(t)=\gamma(t)\cdot \gamma''(t)=-|\gamma'(t)|^2.
\]
Since geodesics have constant speed and \(\gamma'(0)=X\), we get
\[
\gamma''(t)=-|X|^2\gamma(t).
\]
In particular, $\gamma''(0)=-|X|^2\omega$.

Now define a curve in $\R^3$ by $c(t):=\rho\,\gamma(t)$. Then
\[
c(0)=\rho\omega,
\qquad
c'(0)=\rho X,
\qquad
c''(0)=\rho\gamma''(0)=-\rho |X|^2\omega.
\]

On the one hand, by the Euclidean chain rule,
\begin{align}\label{appendcurve}
\frac{d^2}{dt^2}F(c(t))\Big|_{t=0}
&=
D^2F(\rho\omega)[\rho X,\rho X]
+
\nabla F(\rho\omega)\cdot c''(0)\\
&=
\rho^2 D^2F(\rho\omega)[X,X]
-
\rho |X|^2\, \nabla F(\rho\omega)\cdot \omega. \nonumber
\end{align}

Since $F$ is homogeneous of degree $1$, Euler's identity gives
\[
\nabla F(x)\cdot x = F(x)
\qquad
(x\neq 0).
\]
Applying this at $x=\rho\omega$ and using the homogeneity for $F$, we obtain
\[
\nabla F(\rho\omega)\cdot (\rho\omega)=F(\rho\omega)=
\rho F(\omega)=\rho \bar F(\omega).
\]
Therefore
\[
\nabla F(\rho\omega)\cdot \omega = \bar F(\omega).
\]
Substituting this into \eqref{appendcurve} yields
\begin{equation}\label{eq:euclidean-side}
\frac{d^2}{dt^2}F(c(t))\Big|_{t=0}
=
\rho^2 D^2F(\rho\omega)[X,X]
-
\rho |X|^2 \bar F(\omega).
\end{equation}

On the other hand, since $c(t)=\rho\gamma(t)$ and $F$ is $1$-homogeneous,
\[
F(c(t))=F(\rho\gamma(t))=\rho F(\gamma(t))=\rho \bar F(\gamma(t)).
\]
Thus
\[
\frac{d^2}{dt^2}F(c(t))\Big|_{t=0}
=
\rho\,\frac{d^2}{dt^2}\bar F(\gamma(t))\Big|_{t=0}.
\]
Because $\gamma$ is a geodesic on $\mathbb{S}^2$ with initial velocity $X$, we have
\[
\frac{d^2}{dt^2}\bar F(\gamma(t))\Big|_{t=0}
=
\nabla^2_{\mathbb{S}^2}\bar F(\omega)[X,X].
\]
Therefore,
\begin{equation}\label{eq:spherical-side}
\frac{d^2}{dt^2}F(c(t))\Big|_{t=0}
=
\rho\,\nabla^2_{\mathbb{S}^2}\bar F(\omega)[X,X].
\end{equation}

Comparing \eqref{eq:euclidean-side} and \eqref{eq:spherical-side}, we obtain
\[
D^2F(\rho\omega)[X,X]
=
\frac{1}{\rho}\left(\nabla^2_{\mathbb{S}^2}\bar F(\omega)[X,X]+\bar F(\omega)|X|^2\right)
=
\frac{1}{\rho}\,Q^F_\omega[X,X],
\]
and this proves the result.
Applying the lemma to $F=h_K$ yields the desired identity.
\end{proof}

\bibliographystyle{plain}
\bibliography{mybib}

\end{document}